\documentclass[smallextended]{svjour3}

\usepackage{amssymb}
\usepackage{epstopdf,comment}
\usepackage{  mathrsfs, enumerate}
\usepackage{amsmath}

\newtheorem{assumption}{Assumption}

\newcommand{\tr}{\operatorname{tr}}
\newcommand{\Real}{\operatorname{Re}}
\newcommand{\Var}{\operatorname{Var}}

\newcommand{\req}[1]{Eq.\,(\ref{#1})}

\begin{document}

\title{ Uncertainty Quantification for Markov Processes via Variational Principles and Functional Inequalities}

\titlerunning{Uncertainty Quantification for Markov Processes}       

\author{Jeremiah Birrell         \and
        Luc Rey-Bellet
}

\institute{Jeremiah Birrell \at
              Department of Mathematics and Statistics,\\
University of Massachusetts Amherst\\
Amherst, MA, 01003, USA\\              
              \email{birrell@math.umass.edu}           
           \and
           Luc Rey-Bellet \at
Department of Mathematics and Statistics,\\
University of Massachusetts Amherst\\
Amherst, MA, 01003, USA \\
 \email{luc@math.umass.edu}       }

\date{}

\maketitle

\begin{abstract}
Information-theory based variational principles have proven effective at providing scalable uncertainty quantification (i.e. robustness) bounds for quantities of interest in the presence of nonparametric model-form uncertainty.  In this work, we combine such variational formulas with functional inequalities (Poincar{\'e}, $\log$-Sobolev, Liapunov functions) to derive explicit uncertainty quantification bounds for time-averaged observables, comparing a Markov process to a second (not necessarily Markov) process. These bounds are well-behaved in the infinite-time limit and apply to steady-states of both discrete and continuous-time Markov processes. 

\keywords{uncertainty quantification \and Markov process \and relative entropy  \and  Poincar\'e inequality \and $\log$-Sobolev inequality \and Liapunov function \and Bernstein inequality }
 \subclass{ 47D07   \and 39B72 \and 60F10  \and 60J25      }
\end{abstract}

\section{Introduction}
Information-theory based variational principles have proven effective at providing uncertainty quantification (i.e. robustness) bounds for quantities of interest in the presence of nonparametric model-form uncertainty \cite{chowdhary_dupuis_2013,DKPP,doi:10.1063/1.4789612,doi:10.1137/15M1047271,KRW,DKPR,GKRW,BREUER20131552,doi:10.1287/moor.2015.0776,GlassermanXu}. In the present work, we combine these tools with functional inequalities to obtain improved and explicit  uncertainty quantification (UQ) bounds for both discrete and continuous-time Markov processes on general state spaces.

In our approach, we are given a baseline model, described by  a probability measure $P$; this is the model one has `in hand' and that is amenable to analysis/simulation, but it may contain many sources of error and uncertainty.  Perhaps it depends on parameters with uncertain values (obtained from experiment, Monte Carlo simulation, variational inference, etc.) or is obtained via some approximation procedure (dimension reduction, neglecting memory terms, asymptotic approximation, etc.)   In short, any quantity of interest computed from $P$ has (potentially) significant uncertainty associated with it. Mathematically we chose to express this uncertainty by considering a (nonparametric) family, $\mathcal{U}(P)$, of alternative models that we postulate contains the inaccessible `true' model.

 Loosely stated, given some observable  $F$, the  uncertainty quantification goal considered here is
\begin{align}\label{main_problem1}
\text{\bf Bound the bias }  E_{\widetilde{P}}[F]-E_P[F] \text{ where } \widetilde{P} \in \mathcal{U}_r(P).
\end{align}
The subscript $r$ indicates that the `neighborhood' of alternative models, $\mathcal{U}_r(P)$, is often defined in terms of an error tolerance, $r>0$.  For our purposes, the appropriate notion of neighborhood will be expressed in terms of relative entropy, which can be interpreted as measuring the loss of information due to uncertainties. We do not discuss in full generality  how to choose the tolerance level $r$ but there are cases where one has enough information about the `true' model to choose an appropriate tolerance; see Section \ref{sec:rel_ent}.

\begin{remark}
Note that in \req{main_problem1}, and the remainder of this paper, we consider the case where the quantity-of-interest is the expected value of some function, but extensions of these ideas to other quantities of interest are possible \cite{DKPR}.
\end{remark}

 More specifically, here we work with  a  Markov process $(X_t,P^\mu)$ with initial distribution $\mu$ and stationary distribution $\mu^*$, and compare it to an alternative (not necessarily Markov) process $(X_t,\widetilde P^{\widetilde \mu})$; we study the problem of bounding the bias when the finite-time averages of a real-valued observable, $f$, are computed by sampling from the alternative process:
\begin{align}\label{tildeP_bias1}
\text{{\bf Bound \,\,\,\,\,\,\,\,\,\,\,}} \widetilde{E}^{\widetilde{\mu}}\left[\frac{1}{T}\int_0^Tf(X_t)dt\right]-\int f d\mu^*.
\end{align}
Here, $\widetilde{E}^{\widetilde{\mu}}$ denotes the expectation with respect to $\widetilde P^{\widetilde \mu}$ and similarly for $P^\mu$, $E^\mu$. (Discrete-time processes will  also be considered in Section \ref{sec:discrete}.)

\req{tildeP_bias1} is a (less studied) variant of the classical problem of   the convergence of ergodic averages to the expectation in the stationary distribution:
\begin{align}\label{P_ergodic_avg}
E^{\mu}\left[\frac{1}{T}\int_0^Tf(X_t)dt\right]\to  \int f d\mu^*.
\end{align} 

By combining information on the problems \req{tildeP_bias1} and \req{P_ergodic_avg}, one can also obtain bounds on the finite time sampling error:
\begin{align}\label{finite_T_err}
\text{err}_T= E^{\mu}\left[\frac{1}{T}\int_0^Tf(X_t)dt\right]-\widetilde{E}^{\widetilde{\mu}}\left[\frac{1}{T}\int_0^Tf(X_t)dt\right].
\end{align} 
In this work, we focus on the robustness problem, \req{tildeP_bias1}.

There are classical inequalities addressing \req{main_problem1} (ex: Csiszar-Kullback-Pinsker, Le Cam, Scheff{\'e}, etc.), but they exhibit poor scaling properties with problem dimension and/or in the infinite-time limit, and so are inappropriate for bounding \req{tildeP_bias1}. This problem can be addressed by using  tight information inequalities based on the Gibbs variational principle  that are summarized  in Section \ref{sec:UQ_background}. See \cite{KRW} for a detailed discussion of these issues.

Other recent works have also focused on robustness bounds for Markov processes, often with the goal of providing error bounds for approximate Markov chain Monte Carlo samplers. Bounds on the difference between the distributions (finite-time or stationary) of Markov processes have been obtained in both total-variation  \cite{10.2307/30040824,ferre2013,Alquier2016,Medina-Aguayo2016,pmlr-v32-bardenet14,2015arXiv150803387J} and Wasserstein distances \cite{rudolf2018,2014arXiv1405.0182P,2016arXiv160506420H}.

One benefit of the approach  taken in the present work is that the bounds naturally incorporate information on the specific observable, $f$, under investigation; for instance, the asymptotic variance of $f$ under the baseline model appears in the bound in Theorem \ref{thm:poincare_UQ_reversible}, below.  When the end goal is robustness bounds for time-averages of $f$, this observable specificity  has the potential to yield tighter bounds; see also \cite{2014arXiv1405.0182P} for bounds that incorporate similar information on  the observable.

Our method  utilizes relative entropy to quantify the distance between models.  A  drawback, compared to the total-variation and Wasserstein distance approaches, is the requirement of absolute continuity; however, this is satisfied in many cases of interest. As we will see, one benefit of utilizing relative entropy is that  the alternative model does not have to be a Markov process. The second main innovation here is the use of various functional inequalities, in combination with relative entropy, to bound \req{tildeP_bias1}.  The end result is computable, finite-time UQ bounds  that are also well behaved in the long-time limit. 

\subsection{Summary of Results}
The basis for all of our  bounds is Theorem \ref{thm:process_goal_oriented_bound} in Section \ref{sec:UQ_background}:
\begin{align}\label{goal_oriented_bound_intro}
&\pm\left( \widetilde E^{\widetilde\mu}\left[\frac{1}{T}\int_0^T f( X_t) dt\right] - \mu^*[ f]\right)\\
\leq& \inf_{c>0}\left\{\frac{1}{cT}\Lambda_{P^{\mu^*}_T}^{\widehat{f}_T}(\pm c)+\frac{1}{cT}R(\widetilde P^{\widetilde{\mu}}_T||P^{\mu^*}_T)\right\},\,\,\,\,\,\widehat{f}_T\equiv \int_0^T f(X_t)-\mu^*[f]dt,\notag
\end{align}
along with Corollary \ref{thm:master_UQ_bound} in Section \ref{sec:Kac}:
 \begin{align}\label{master_bound_summary}
&\frac{1}{T}\Lambda_{P^{\mu^*}_T}^{\widehat{f}_T}(\pm c)\leq \kappa(V_{\pm c}),\\
&\kappa(V)\equiv\sup\left\{\langle A[g],g\rangle+\int V|g|^2d\mu^*:g\in D(A,\mathbb{R}),\,\|g\|_{L^2(\mu^*)}=1\right\}, \label{eq:varbound}\\
&V_{\pm c}(x)\equiv \pm c\left(f(x)-\mu^*[f]\right),\,\,\,\mu^*[f]\equiv\int fd\mu^*.
\end{align}
In the above, $\Lambda_{P^{\mu^*}_T}^{\widehat{f}_T}(\pm c)$ is the cumulant generating function of $\widehat{f}_T$ (see \req{Lambda_formula} for details), $R(\widetilde P^{\widetilde{\mu}}_T||P^{\mu^*}_T)$ is the relative entropy of the processes up to time $T$ (see \req{rel_ent_def}),  $\langle\cdot,\cdot\rangle$ denotes the inner product on $L^2(\mu^*)$, and $(A,D(A,\mathbb{R}))$ is the generator of the Markov semigroup for the process $(X_t,P^\mu)$ on $L^2(\mu^*)$. Again, we emphasize that the alternative process, $(X_t,\widetilde{P}^{\widetilde{\mu}})$, does {\em not} need to be Markov; for an example involving semi-Markov processes, see Section \ref{Example:semi-Markov}.

\req{goal_oriented_bound_intro} is  derived by employing the Gibbs variational principle (hence the relation to relative entropy). \req{master_bound_summary}, which is based on a theorem proven in \cite{WU2000435},  results from a connection between the cumulant generating function and the Feynman-Kac semigroup (hence the appearance of the generator,  $A$).  Also, note that the bound is expected to behave well in the limit $T\to \infty$, as $R(\widetilde P^{\widetilde{\mu}}_T||P^{\mu^*}_T)/T$ converges to the  relative entropy rate of the processes, under suitable ergodicity assumptions.

\req{master_bound_summary} allows us to employ our   primary new tool for UQ, that is, functional inequalities.  By functional inequalities, we mean bounds on the generator, $A$, that will yield bounds on $\kappa(V_{\pm c})$; we will cover Poincar{\'e}, $\log$-Sobolev, and $F$-Sobolev inequalities, as well as Liapunov functions. Our results rely heavily on the bounds obtained  in \cite{WU2000435,lezaud:hal-00940906,cattiaux_guillin_2008,Guillin2009,doi:10.1137/S0040585X97986667} where 
concentration inequalities for ergodic averages were obtained.

 The method outlined above leads to  explicit UQ bounds, expressed in terms of the following quantities:
 \begin{enumerate}
 \item Spectral properties of the generator, $A$, of the dynamics of the baseline model, $P$, in the stationary regime (i.e., on $L^2(\mu^*)$); see \req{goal_oriented_bound_intro}-\req{eq:varbound}. This term depends on the chosen observable but does {\em not} depend on the alternative model; functional inequalities are only required for the base model (which is often the simpler model).  This is one of the strengths of the method, though  computing explicit, tight constants for these functional inequalities is  still a very difficult problem in general. 
 \item The path-space relative entropy up to time $T$, $R(\widetilde P^{\widetilde{\mu}}_T||P^{\mu^*}_T)$, of the alternative model with respect to the base. This term depends heavily  on the difference in dynamics between the two models; in particular, a nontrivial bound  requires absolute continuity of the path-space distributions.  This is a drawback of the relative-entropy based method employed here, but it does hold in many cases of interest; see Section \ref{sec:rel_ent} for examples.   
 \end{enumerate}
While most of our results do not assume reversibility of the base process, bounds based on \req{eq:varbound} only involve the symmetric part of the generator and so are generally  less than ideal, or even useless, for many nonreversible systems.  This is a drawback of the approach pursued here.
 \begin{remark}
 For certain hypocoercive systems,  ergodicity can be proven by working with an alternative metric;  see \cite{Dolbeault2009,Dolbeault2015,2018arXiv180804299D}.  It is possible that the functional-inequality based UQ techniques developed below could be adapted to this more general setting; a step in that direction can be found in \cite{2019arXiv190711973B}.
\end{remark}

 For a simple example of the type of result obtained below, consider diffusion on $\mathbb{R}^n$ in a $C^2$ potential, $V$, i.e. the generator is $A=\Delta-\nabla V\cdot V$ and the invariant measure is $\mu^*=e^{-V(x)}dx$.  Suppose the Hessian of $V$ is  bounded below:
\begin{align}\label{hessian_bound}
D^2V(x)\geq \alpha^{-1} I,\,\,\,\alpha>0.
\end{align}
Our results give a Bernstein-type UQ bound for any bounded $f$:
\begin{align}\label{hessian_UQ_bound}
&\pm \left(\widetilde E^{\widetilde \mu}\left[\frac{1}{T}\int_0^Tf(X_t)dt\right]-\mu^*[f]\right)\leq \sqrt{2\sigma^2\eta}+M^{\pm} \eta,\\
&M^\pm =\alpha\|(f-\mu^*[f])^\pm\|_\infty,\,\,\,\sigma^2=2\alpha\Var_{\mu^*}[f],\,\,\,\eta=\frac{1}{T}R(\widetilde{P}^{\widetilde\mu}_T||P^{\mu^*}_T).\notag
\end{align}
(This bound can also be improved by using the asymptotic variance; see Section \ref{sec:reversible}.)  Section \ref{sec:log_sob_diffusion} contains further discussion of diffusions.

The remainder of the paper is structured as follows. Necessary background on UQ for both general measures and processes will be given in Section \ref{sec:UQ_background}, leading up to a connection with both the Feynman-Kac semigroup and  the relative entropy rate.  Relevant properties of the Feynman-Kac semigroup are given  in Section \ref{sec:Kac}, culminating with the bound \req{master_bound_summary}. The use of functional inequalities to obtain explicit UQ bounds from \req{master_bound_summary} will be explored in Section \ref{sec:functional_ineq}.  In Section \ref{sec:discrete} we show how these ideas can be adapted to discrete-time processes.    Finally, the problem of bounding the relative entropy rate will be addressed in Section \ref{sec:rel_ent}.

\section{Uncertainty Quantification for Markov Processes}\label{sec:UQ_background}

\subsection{UQ via Variational Principles}  
In this subsection, we recall several earlier results regarding the variational-principle approach to UQ, as developed in \cite{chowdhary_dupuis_2013,DKPP,BREUER20131552,GlassermanXu,BLM}.  In particular, Proposition \ref{thm:goal_div}, quoted from \cite{DKPP}, will be a critical tool in our derivation of UQ bounds for Markov processes.

Let $P$ be a probability measure on a measurable space $(\Omega,\mathcal{F})$.   We consider the class of random 
variables $f: \Omega \to \mathbb{R}$  with a well-defined and finite moment generating function in a neighborhood of the origin: 

\begin{equation}
\mathcal{E}(P)  \,=\,   \left\{  f:  \Omega \to \mathbb{R} \,:\,  E_P[e^{\pm c_0f}]<\infty \textrm{ for some } c_0 >0 \right\}\,. 
\end{equation}
It is not difficult to prove (see e.g. \cite{DemboZeitouni}) that the cumulant generating function 
\begin{equation}
\Lambda_P^f(c) = \log E_P[e^{ c f}]
\end{equation} 
is a convex function, finite and infinitely differentiable  in some interval $(c_{-},c_{+})$  with $-\infty \le c_{-} < 0 < c_+ \le \infty$  
and equal to $+\infty$ outside of $[c_{-},c_{+}]$.   Moreover if $f \in \mathcal{E}(P)$ then $f$ has moments of 
all orders and we write 
\begin{equation}
\widehat{f} = f - E_P[f]\,
\end{equation} 
for the centered observable of mean $0$.  We will often use the cumulant generating function for the centered observable $\widehat{f}$:
\begin{equation}
\Lambda_P^{\widehat{f}}(c) = \log E_P[e^{ c (f - E_P[f])} ] =  \Lambda_P^f(c) - c E_P[f] \,.
\end{equation}

Recall also the relative entropy (or Kullback-Leibler divergence), defined by 
\begin{equation}\label{rel_ent_def}
R( \widetilde{P} || P)= \left\{ \begin{array}{cl}  E_{\widetilde{P}} \left[\log\left(\frac{d\widetilde{P}}{dP}\right)\right]  & \textrm{ if } \widetilde{P} \ll P  \\  + \infty & \textrm{ otherwise}
\end{array} \,.
\right. 
\end{equation}
It has the property of a divergence, that is $R( \widetilde{P} || P) \ge 0$  and $R( \widetilde P || P) = 0$ if and only if $\widetilde{P}=P$.

A key ingredient in our approach is the Gibbs variational principle which relates the cumulant generating function and relative entropy; see Proposition 1.4.2 in \cite{dupuis2011weak}.
\begin{proposition}[\bf Gibbs Variational Principle]\label{prop:Gibbs} Let $f:\Omega\to\mathbb{R}$ be bounded and measurable.  Then
\begin{equation}\label{gibbs}
\log E_P[e^f] = \sup_{\widetilde{P}: R(\widetilde{P}||P)< \infty} \left\{  E_{\widetilde{P}}[f]  - R(\widetilde{P} || P) \right\}.
\end{equation}  
\end{proposition}

As shown in \cite{chowdhary_dupuis_2013,DKPP}, the Gibbs variational principle  implies the following UQ bounds for the expected values:  (a similar inequality is used in the context of concentration inequalities, see e.g. \cite{BLM}, and was also used independently in  \cite{BREUER20131552,GlassermanXu}). For a proof of the version stated here, see pages 85-86  in \cite{DKPP}.
\begin{proposition}[\bf Gibbs information inequality]\label{thm:goal_div} If $R(\widetilde{P} || P)<\infty$ and  $f\in\mathcal{E}(P)$ then $f\in L^1(\widetilde{P})$ and
\begin{align}\label{goal_oriented_bound}
- \inf_{c >0}  \left\{  \frac{\Lambda_P^{\widehat{f}}(-c)+R(\widetilde{P} ||P)}{c} \right\} 
 \leq E_{\widetilde P}[f]-E_P[f]\leq  
\inf_{c >0}  \left\{  \frac{\Lambda_P^{\widehat{f}}(c)+R(\widetilde{P} ||P)}{c} \right\} \,.
\end{align}
\end{proposition}

\begin{remark}
Note that  even if $R(\widetilde{P}||P)=\infty$, the bound \req{goal_oriented_bound} trivially holds as long as $E_{\widetilde P}[f]$ is defined.  To avoid clutter in the statement of our results, when $R(\widetilde{P}||P)=\infty$ we will consider the bound  to be  satisfied for any $f\in\mathcal{E}(P)$, even if $E_{\widetilde P}[f]$ is undefined.
\end{remark}

Optimization problems  of the form  in \req{goal_oriented_bound} will appear frequently, hence we make the following definition:
\begin{definition}
Given any $\Lambda:\mathbb{R}\to[0,\infty]$ and $\eta>0$, let
\begin{align}
\Xi^\pm(\Lambda,\eta)\equiv  \inf_{c>0}\left\{\frac{\Lambda(\pm c) +\eta}{c}\right\}.
\end{align}
\end{definition}
With this, we can  rewrite the bound \eqref{goal_oriented_bound} as 
\begin{equation}\label{goal_oriented_bound2}
- \Xi^- \left(\Lambda^{\widehat{ f}}_P,R(\widetilde{P}||P)\right) \leq E_{\widetilde{P}}[f]-E_P[f] \leq \Xi^+ \left(\Lambda^{\widehat{f}}_P,R(\widetilde{P}||P)\right)\,.
\end{equation}

\req{goal_oriented_bound2} is the starting point for all   UQ bounds derived in this paper. From  it, we see which quantities must be controlled in order  to make the UQ bounds explicit: the relative entropy and the cumulant generating function.  The former will be discussed in Section \ref{sec:rel_ent}. For Markov processes, the latter can be bounded via a connection with the Feynman-Kac semigroup and functional inequalities; this connection between functional inequalities and UQ bounds is the main focus and innovation of the current work, and we begin discussing it in Section \ref{sec:UQ_Kac}.  First we recall some general properties of the bounds (\ref{goal_oriented_bound2}).

\subsection{Properties of $\Xi^\pm$}
The objects
\begin{align}
\Xi(\widetilde{P}||P;\pm f)\equiv\Xi^\pm \left(\Lambda^{\widehat{ f}}_P,R(\widetilde{P}||P)\right)
\end{align}
appearing in the Gibbs information inequality, \req{goal_oriented_bound2},  have many remarkable properties, of which we recall a few.  
\begin{proposition}\label{prop:div} 
Assume R($\widetilde{P} || P) < \infty$ and  $f\in\mathcal{E}(P)$.  We have:
\begin{enumerate}
\item{\bf (Divergence)} $\Xi(\widetilde{P}||P;f)$ is a divergence,  i.e. $\Xi(\widetilde{P}||P,f) \ge 0$ and\\
 $\Xi(\widetilde{P}||P;f) = 0$ if and only if either $P=\widetilde{P}$  or $f$ is constant $P$ a.s. 
\item{\bf (Linearization)}  If  $R(\widetilde{P}||P)$ is sufficiently small we have  
\begin{equation}\label{Xi_linearization}
\Xi(\widetilde{P}||P,f) =  \sqrt{ 2 {\rm Var}_P[f] R(\widetilde{P}||P)} + O (R(\widetilde{P}||P))\,.
 \end{equation}
\item{\bf (Tightness)}  For $\eta >0$ consider $\mathcal{U}_\eta = \{ \widetilde{P} \,;\, R(\widetilde{P}||P) \le \eta\}$.  There exists $\eta^*$ 
with $0  < \eta^* \le  \infty$ such that for any $\eta < \eta^*$ there exists a measure $P^\eta$ with 
\begin{eqnarray}
\sup_{ \widetilde{P} \in \mathcal{U}_\eta}  \left\{ E_{\widetilde{P}}[f]- E_P[f] \right\} =  E_{ P^\eta} [f] - E_P[f]  \,=\,  \Xi(P^\eta ||P;f) \,.
\end{eqnarray} 
The measure $P^\eta$ has the form 
\begin{equation}
d P^\eta = e^{ c f - \Lambda^f_P(c)} dP,
\end{equation}
where $c=c(\eta)$ is the unique nonnegative solution of  $
R( P^\eta || P ) = \eta$.
\end{enumerate}
\end{proposition}

\begin{proof}
 Items 1 and 2 are proved in \cite{DKPP}; see also \cite{doi:10.1287/moor.2015.0776} for item 2.  Various versions of the proof of item 3  can be found in  in \cite{chowdhary_dupuis_2013} or \cite{DKPP}.  See Proposition 3 in \cite{DKPR}  for a more detailed statement of the result; see also similar results  in \cite{BREUER20131552,doi:10.1111/mafi.12050}.\qed
\end{proof}

The tightness property in Proposition \ref{prop:div} is very attractive and ultimately relies on the presence of the  
cumulant generating function  $\Lambda^{\widehat{f}}_P(c)$, which encodes the entire law of $f$. However, this generally this makes the bound very difficult or impossible to compute explicitly; we will  need to weaken \req{goal_oriented_bound2}  to obtain more usable bounds. Functional inequalities are one tool we will employ (see Section \ref{sec:functional_ineq}).  Another ingredient, which we discuss next, will be explicit bounds on the optimization problem in the definition of $\Xi^\pm(\Lambda,\eta)$. Such an approach was put forward in \cite{GKRW} where various concentration inequalities such as Hoeffding, sub-Gaussian, and Bennett bounds are discussed.  For this paper we will almost exclusively use  the following Bernstein-type bound:
\begin{lemma}\label{lemma:Bernstein} 
Suppose there exist $\sigma>0$, $M^\pm\geq 0$ such that
\begin{align}\label{lambda_Bernstein_bound}
\Lambda(\pm c)\leq \frac{\sigma^2c^2}{2(1-cM^\pm)}
\end{align}
for all $0<c<1/M^\pm$.  Then for all $\eta\geq 0$ we have
\begin{align}\label{Bernstein_inf_bound}
\Xi^\pm(\Lambda,\eta)\leq \sqrt{2\sigma^2\eta}+M^\pm\eta.
\end{align}
Note that  $M^\pm=0$ covers the case of a (one-sided) sub-Gaussian concentration bound.
\end{lemma}
\begin{proof}
Bound $\Lambda$ using \req{lambda_Bernstein_bound} and solve the resulting optimization problem on $0<c<1/M^\pm$.
\qed
\end{proof}

From the point of view of concentration inequalities, the bound \req{lambda_Bernstein_bound} is not very tight; indeed it holds for the 
cumulant generating function $\Lambda^{\widehat{ f}}_P$ of  any random variable $f \in \mathcal{E}(P)$, but explicit constants may be 
hard to come by.  In the context of Markov process it has however been proved to be extremely useful, see \cite{WU2000435,lezaud:hal-00940906,cattiaux_guillin_2008,Guillin2009} and in particular \cite{doi:10.1137/S0040585X97986667}.

Second, we will need a linearization bound, generalizing \req{Xi_linearization}:
\begin{lemma}\label{lemma:asymp_bound}
Let $\Lambda:\mathbb{R}\to[0,\infty]$ be $C^2$ on a neighborhood of $0$, $\Lambda(0)=\Lambda^\prime(0)=0$, and $\Lambda^{\prime\prime}(0)>0$.  Then
\begin{align}\label{asymp_bound}
\inf_{c>0}\left\{\frac{\Lambda(\pm c)+\eta}{c}\right\}\leq \sqrt{2\Lambda^{\prime\prime}(0)\eta}+o(\sqrt{\eta})
\end{align}
as $\rho\searrow 0$. If $\Lambda^{\prime\prime}$ is Lipschitz at $0$ then the error bound improves to $O(\eta)$.
\end{lemma}
\begin{proof}
The bound follows from Taylor expansion of $\Lambda(c)$; see the proof  of Theorem 2.8 in \cite{DKPP}.
\qed
\end{proof}

\subsection{UQ for Markov Processes}\label{sec:UQ_processes}
One of the main advantages of the Gibbs information inequality, \req{goal_oriented_bound}, over  classical information inequalities (such as the Kullback-Leibler-Czisz\`ar inequality) is how it scales with time when applied to the distributions of processes on path space.  See  \cite{KRW} for a detailed discussion of this issue.  This strength will become apparent as we proceed.

The following assumption details the setting in which we will work for the remainder of this paper:
\begin{assumption}\label{general_process_assump}
Let $\mathcal{X}$ be a Polish space  and suppose we have a  time homogeneous, $\mathcal{X}$-valued, c\`adl\`ag Markov family  $(\Omega,\mathcal{F},\mathcal{F}_t,X_t,P^x)$, $x\in\mathcal{X}$, with transition probability kernel $p_t$ (see the statement of Theorem \ref{theorem:markov_family} in Appendix \ref{app:jump} for the precise definition of a Markov family that we use).

Also assume we have a second probability kernel  $\widetilde P^x$, $x\in\mathcal{X}$,  on $(\Omega,\mathcal{F})$  with $({X}_0)_*\widetilde P^x=\delta_x$ for each $x\in\mathcal{X}$.
\end{assumption}
\begin{remark}
We are {\em not} assuming $(X_t,\widetilde P^x)$ are Markov processes.
  \end{remark}

  One of of the families, $P^x$ or $\widetilde{P}^x$, is thought of as the base model, and the other as some alternative (or approximate) model, but which is which can vary  with the application. From a mathematical perspective, the primary factors distinguishing $P^x$ and $\widetilde{P}^x$ are:
  \begin{enumerate}
  \item Our methods require  information on the spectrum of the generator of the semigroup associated with $p_t$.
  \item $(X_t, P^x)$ must be Markov, but $(X_t,\widetilde P^x)$ can be non-Markovian. 
  \end{enumerate}
    $P^x$ and $\widetilde P^x$ should be chosen with these points in mind; in the remainder of this paper, we will refer to the former as the base model and the latter as the alternative model.

\begin{definition}\label{def:PT}

Given initial distributions $\mu$ and $\widetilde\mu$ on $\mathcal{X}$, we also define the probability measures
\begin{align}
P^{\mu}(\cdot)=\int P^x(\cdot)\mu(dx),\,\,\widetilde P^{\widetilde \mu}(\cdot)=\int \widetilde P^x(\cdot)\widetilde \mu(dx).
\end{align}
Note that Assumption \ref{general_process_assump} implies  that $X_t$ is a Markov process for the space $(\Omega,\mathcal{F}_t, P^{\mu})$ with initial distribution $\mu$ and time-homogeneous transition probabilities $p_t$. 

We will also need the finite-time restrictions, which can be thought of as the distributions on path space up to some $T>0$:
\begin{align}\label{path_measure}
P^x_T\equiv P^x|_{\mathcal{F}_T},\,\, \widetilde P^x_T\equiv \widetilde P^x|_{\mathcal{F}_T},
\end{align}
and similarly for $P^{\mu}_T$ and $\widetilde{P}^{\widetilde{\mu}}_T$. Finally, we let $E^\mu$ denote the expected value with respect to $P^\mu$ and similarly for $\widetilde{E}^{\widetilde\mu}$.
\end{definition}

Now fix a bounded measurable $f:\mathcal{X}\to\mathbb{R}$ (the boundedness assumption will be relaxed later) and an invariant measure $\mu^*$ for $p_t$. As mentioned in the introduction, there are many classical techniques for studying  convergence of the ergodic averages of $f$ under $P^\mu$ to the average in the invariant measure, $\mu^*[f]$. Therefore, in this paper we consider the much less studied problem of bounding the bias when the finite-time averages are computed by sampling from the alternative distribution; see \req{tildeP_bias1}.

\subsection{UQ Bounds via the Feynman-Kac Semigroup}\label{sec:UQ_Kac}

Due to  our interest in the problem (\ref{tildeP_bias1}), we start the $P$-process in the invariant distribution $\mu^*$, while the $\widetilde P$-process is started in an arbitrary distribution $\widetilde\mu$.

Given a bounded measurable function   $f$ on $\mathcal{X}$ and $T>0$, define the  bounded and $\mathcal{F}_T$-measurable function 
  \begin{align}
  f_T=\int_0^T f(X_t) dt.
\end{align}
 Applying the Gibbs information inequality, \req{goal_oriented_bound}, to $f_T$, $\widetilde{P}^{\widetilde\mu}_T$, $P^{\mu^*}_T$ and dividing by $T$ yields:
 \begin{theorem}\label{thm:process_goal_oriented_bound}
 \begin{align}\label{goal_oriented_bound3}
&\pm\left( \widetilde E^{\widetilde\mu}\left[\frac{1}{T}\int_0^T f( X_t) dt\right] - \mu^*[ f]\right)\leq \Xi^\pm\left(\frac{1}{T}\Lambda_{P^{\mu^*}_T}^{\widehat{f}_T},\frac{1}{T}R(\widetilde P^{\widetilde{\mu}}_T||P^{\mu^*}_T)\right),
\end{align}
where 
\begin{align}
\mu^*[f]\equiv\int fd\mu^*,\,\,\,\,\widehat{f}_T\equiv \int_0^T f(X_t)-\mu^*[f]dt.
\end{align}
\end{theorem}
\begin{remark}
Recall the definition
\begin{align}\label{Xi_def}
\Xi^\pm(\Lambda,\eta)=  \inf_{c>0}\left\{\frac{\Lambda(\pm c) +\eta}{c}\right\}.
\end{align}
All of the UQ bounds we obtain will be of the form
 \begin{align}\label{standard_goal_oriented_bound}
&\pm \left(\widetilde E^{\widetilde{\mu}}\left[\frac{1}{T}\int_0^T f( X_t) dt\right]-\mu^*[f]\right)
\leq \Xi^\pm(\Lambda,\eta)
\end{align}
for some  $\Lambda:\mathbb{R}\to[0,\infty]$ and  $\eta>0$; we will refer back to these equations often.
\end{remark}

To produce a more explicit  bound from \req{goal_oriented_bound3}, one needs to bound the cumulant generating function as well as the relative entropy.  The latter will be addressed in Section \ref{sec:rel_ent}. As for the former, observe that  the cumulant generating function can be written
\begin{align}\label{Lambda_formula}
\Lambda_{P_T^{\mu^*}}^{\widehat{f}_T}(\pm c)=&\log\left( \int E^x\left[\exp\left(\pm c\int_0^Tf(X_t)- \mu^*[f]dt\right)\right]\mu^*(dx)\right).
\end{align}
 \req{Lambda_formula} is  related to the Feynman-Kac semigroup on $L^2(\mu^*)$ with potential $V$:
 \begin{align}\label{eq:Kac_def}
\mathcal{P}_t^V[g](x)=E^x\left[g(X_t)\exp\left(\int_0^t V(X_s)ds\right)\right].
\end{align} 

More specifically,
\begin{align}\label{Lambda_bound}
\Lambda_{P_T^{\mu^*}}^{\widehat{f}_T}(\pm c)\leq&\log\left( \|\mathcal{P}_T^{V_{\pm c}}[1]\|_{L^2(\mu^*)}\right),\\
V_{\pm c}(x)\equiv& \pm c\left(f(x)-\mu^*[f]\right),
\end{align}
and so we obtain:
\begin{lemma}
Under Assumption  \ref{general_process_assump},  for any bounded measurable $f:\mathcal{X}\to\mathbb{R}$, \req{standard_goal_oriented_bound} holds with 
\begin{align}\label{Kac_UQ_bound}
\Lambda(\pm c)= \frac{1}{T}\log\left( \|\mathcal{P}_T^{V_{\pm c}}[1]\|_{L^2(\mu^*)}\right),\,\,\,\eta=\frac{1}{T}R(\widetilde P^{\widetilde{\mu}}_T||P^{\mu^*}_T).
\end{align}
\end{lemma}
In the following two sections, we  discuss how  functional inequalities can be used to obtain more explicit bounds on the norm of the Feynman-Kac semigroup.

\section{Bounding the Feynman-Kac Semigroup}\label{sec:Kac}
 The Lumer-Phillips theorem (a variant of the Hille-Yosida theorem) is our  tool of choice for bounding the Feynman-Kac semigroup; see Chapter IX, p. 250 in \cite{Yosida1971} or  Corollary 3.20 in Chapter II of \cite{engel2006short}.  This is the same strategy used in \cite{WU2000435,cattiaux_guillin_2008,doi:10.1137/S0040585X97986667} to obtain concentration inequalities.

First we  state some of the basic properties of the Feynman-Kac semigroup, adapted from  \cite{WU2000435,cattiaux_guillin_2008}.
\begin{proposition}\label{thm:kac}
Let  $V:\mathcal{X}\to\mathbb{R}$ be bounded and measurable and $\mu^*$ be an invariant probability measure for $p_t$.  The operators $\mathcal{P}^V_t$,  $t\geq 0$, defined in \req{eq:Kac_def}, are  bounded linear operators on $L^2(\mu^*)$ that form a strongly continuous semigroup.

If $(A,D(A))$ denotes the generator of $\mathcal{P}_t\equiv \mathcal{P}_t^0$ on $L^2(\mu^*)$ then the generator of $\mathcal{P}_t^V$ on $L^2(\mu^*)$ is $(A+V,D(A))$.
\end{proposition}
\begin{remark}
$D(A)$ consists of complex-valued functions.  We will use  $D(A,\mathbb{R})$ to denote the real-valued functions in the domain of $A$.
\end{remark}

 To bound the norm of the Feynman-Kac semigroup, we use the following Hilbert space version of the Lumer-Phillips theorem (again, see \cite{Yosida1971,engel2006short}, as well as Theorem II.3.23 in \cite{engel2006short} for a proof that \req{eq:A_spectrum_bound} implies $A-\alpha$ is dissipative):
\begin{proposition}\label{thm:Lumer-Phillips}
Let $H$ be a Hilbert space and $Q(t)$ be a strongly continuous semigroup on $H$ with generator $(A,D(A))$.  Suppose that there is an $\alpha\in\mathbb{R}$ such that
\begin{align}\label{eq:A_spectrum_bound}
\Real(\langle Ax,x\rangle)\leq \alpha 
\end{align}
for all $x\in D(A)$ with $\|x\|=1$.  Then $\|Q(t)\|\leq e^{\alpha t}$ for all $t\geq 0$. 
\end{proposition}

 Propositions \ref{thm:kac} and \ref{thm:Lumer-Phillips} together yield a bound on the Feynman-Kac semigroup, in terms of the generator; this result, and generalizations, were proven in \cite{WU2000435} (see Case I in the proof of Theorem 1).
\begin{proposition}\label{corr:kac_bound}
Let $V:\mathcal{X}\to\mathbb{R}$ be bounded and measurable, and for $t\geq 0$ consider the Feynman-Kac semigroup $\mathcal{P}_t^V:L^2(\mu^*)\to L^2(\mu^*)$ with generator $(A+V,D(A))$.

Define
\begin{align}
\kappa(V)=&\sup\left\{\Real(\langle (A+V)[g],g\rangle):g\in D(A), \|g\|_{L^2(\mu^*)}=1\right\}\label{kappa_def1}\\
=&\sup\left\{\langle A[g],g\rangle+\int V|g|^2d\mu^*:g\in D(A,\mathbb{R}),\,\|g\|_{L^2(\mu^*)}=1\right\}\label{kappa_def2},
\end{align}
where $\langle\cdot,\cdot\rangle$ denotes the inner product on $L^2(\mu^*)$.

 Then   the operator norm satisfies the bound
\begin{align}\label{kac_norm_bound}
\|\mathcal{P}_t^V\|\leq e^{t\kappa(V)}
\end{align}
for all $t\geq 0$.

\end{proposition}

Combining \req{kac_norm_bound} with \req{Lambda_bound} and \req{goal_oriented_bound3}, we  obtain UQ bounds that are expressed in terms of the generator of the dynamics of the baseline process and the relative entropy of the alternative process with respect to the base:
\begin{corollary}\label{thm:master_UQ_bound}
Under Assumption  \ref{general_process_assump},  for any bounded measurable $f:\mathcal{X}\to\mathbb{R}$, the UQ bound \req{standard_goal_oriented_bound} holds with 
\begin{align}
\Lambda(\pm c)= \kappa(V_{\pm c}),\,\,\,\,\eta=\frac{1}{T}R(\widetilde P^{\widetilde{\mu}}_T||P^{\mu^*}_T).
\end{align}
\end{corollary}
From \req{standard_goal_oriented_bound} we see that functional inequalities, by which we mean bounds on the generator $A$ that lead to bounds on $\kappa(V_{\pm c})$, can be used to produce UQ bounds. Also, note that the only remaining $T$-dependence is in the relative entropy term, $R(\widetilde P^{\widetilde{\mu}}_T||P^{\mu^*}_T)/T$.  This will often have a finite limit (the relative entropy rate) as $T\to\infty$;  for examples, see Section \ref{Example:semi-Markov}, as well as \cite{10.2307/3216060},  the supplementary materials to \cite{DKPP}, and Appendix 1 of \cite{kipnis2013scaling}. Hence Corollary \ref{thm:master_UQ_bound} shows that one can expect UQ bounds that are well behaved as $T\to\infty$.

\begin{remark}
 Proposition \ref{corr:kac_bound} is stated for bounded $V$, but it can be extended to certain unbounded $V$ under the additional assumption that the symmetrized Dirichlet form is closable;  see Theorem 1 in \cite{WU2000435}.  However, as noted in Corollary 3 in this same reference (and outlined in Proposition \ref{Kac_log_sobolev} below), that assumption can be avoided in the presence of functional inequalities by working with bounded $V$ and then taking limits; this is the strategy we employ here. 
\end{remark}

\section{ UQ Bounds From Functional Inequalities}\label{sec:functional_ineq}

In this section, we explore the consequences of several important classes of  functional inequalities: Poincar\'e, $\log$-Sobolev, and Liapunov functions. Discussion of $F$-Sobolev inequalities, a generalization of the classical $\log$-Sobolev case, can be found in Appendix \ref{app:F_sobolev}. 
\subsection{Poincar\'e Inequality}
First we consider the case where the generator satisfies a Poincar\'e inequality with constant $\alpha>0$, meaning:
\begin{align}\label{Poincare_def}
\Var_{\mu^*}[g]\leq -\alpha \langle A[g],g\rangle
\end{align}
for all $g\in D(A,\mathbb{R})$.  This can equivalently be written
\begin{align}\label{Poincare_def2}
\Real(\langle A[g],g\rangle)\leq -\alpha^{-1}\|P^\perp g\|^2
\end{align}
for all $g\in D(A)$, where $P^\perp$ is the orthogonal projector onto $1^\perp$.

In the presence of a Poincar{\'e} inequality, Proposition \ref{corr:kac_bound} is most useful when combined with the following perturbation result. A version of this result is contained in \cite{WU2000435}, but we present it here in a slightly more general form.  The proof is given in Appendix \ref{app:Hilbert_space_bounds}.
\begin{lemma}\label{lemma:perturb_bound}
Let $H$ be  a Hilbert space, $A:D(A)\subset H\to H$ be a linear operator, and $B:H\to H$ be a bounded self-adjoint operator.  Suppose there exist $D>0$ and $x_0\in H$ with $\|x_0\|=1$ such that
\begin{align}
\langle Bx_0,x_0\rangle=0\,\,\,\,\,\text{ and }\,\,\,\,\,
\Real(\langle Ax,x\rangle)\leq -D\|P^\perp x\|^2
\end{align}
for all $x\in D(A)$,  where $P^\perp$ is the orthogonal projector onto $x_0^\perp$.

Define
\begin{align}
B^+\equiv \max\left\{\sup_{\|y\|=1}\langle By,y\rangle,0\right\}.
\end{align}
Then for any  $0\leq c <D/B^+$ we have
\begin{align}\label{perturb_bound}
\sup_{x\in D(A),\|x\|=1}\Real(\langle (A+c B)x,x\rangle)\leq\frac{c^2\|Bx_0\|^2}{D-c B^+}.
\end{align}

\end{lemma}
\begin{remark}
The above lemma applies to a general Hilbert space. In this paper, we will apply it to $H=L^2(\mu^*)$ (with the associated $L^2$-inner product), and $x_0=1$ (constant function), in which case $P^\perp[f]=f-\mu^*[f]$.
\end{remark}

The multiplication operator by $V_{\pm 1}$ is a bounded self-adjoint operator and $\langle V_{\pm 1}1,1\rangle=0$.  Therefore  Lemma \ref{lemma:perturb_bound} implies:
\begin{lemma}
For all $0\leq c<1/ (\alpha\|(f-\mu^*[f])^\pm\|_\infty)$ we have
  \begin{align}\label{poincare_kappa_bound}
\kappa(V_{\pm c})\leq \frac{\alpha \Var_{\mu^*}[f]c^2}{1- \alpha\|(f-\mu^*[f])^\pm\|_\infty c}.
\end{align}
\end{lemma}

From this, combined with Corollary \ref{thm:master_UQ_bound}, we obtain the following UQ bound:
 \begin{theorem}\label{thm:poincare_UQ}
Under Assumption \ref{general_process_assump}, if $A$ satisfies the Poincar{\'e} inequality, \req{Poincare_def}, then for any bounded measurable $f:\mathcal{X}\to\mathbb{R}$  the bounds \req{standard_goal_oriented_bound} and \req{Bernstein_inf_bound} hold with
\begin{align}
M^\pm=\alpha \|(f-\mu^*[f])^\pm\|_\infty,\,\,\,\sigma^2=2\alpha \Var_{\mu^*}[f],\,\,\, \eta=\frac{1}{T}R(\widetilde P^{\widetilde{\mu}}_T||P^{\mu^*}_T).
\end{align}

\end{theorem}

\subsection{Poincar{\'e} Inequality for  Reversible Processes }\label{sec:reversible}
When the combination of $\mu^*$ and $p_t$ are reversible, i.e. the generator $A$ is self-adjoint on $L^2(\mu^*)$, and if a  Poincar\'e inequality, \req{Poincare_def}, also holds with constant $\alpha>0$, then one can obtain a UQ bound in terms of the asymptotic variance of $f$, instead of the variance of $f$ under $\mu^*$.

First, define the Poisson operator
\begin{align}
L:f\to \int_0^\infty \mathcal{P}_t[f] dt,
\end{align}
a bounded linear operator on $L^2_0(\mu^*)\equiv\{f\in L^2(\mu^*):\mu^*[f]=0\}$ with norm bound $\|L\|\leq\alpha$. The asymptotic variance of $f\in L^2(\mu^*,\mathbb{R})$ is defined by
\begin{align}\label{asymp_var_def}
\sigma^2(f)\equiv\langle 2L[f-\mu^*[f]],f-\mu^*[f]\rangle=2\int_0^\infty\left( \int \mathcal{P}_t[f]fd\mu^*-\left(\mu^*[f]\right)^2\right) dt.
\end{align}
Note that  $0\leq\sigma^2(f)\leq 2\alpha\Var_{\mu^*}[f]$.

Using these objects, one can obtain the following Bernstein-type bound.   A simple proof  appears below Remark 2.3 in \cite{doi:10.1137/S0040585X97986667}; we outline the essential ideas below.  See \cite{lezaud:hal-00940906} and  \cite{Guillin2009} for similar earlier results.
 \begin{lemma}\label{lemma:Bernstein_bound_symmetric}
 For all $0<c<1/(\alpha\|(f-\mu^*[f])^\pm\|_\infty)$ we have
\begin{align}\label{Bernstein_bound_symmetric}
\kappa(V_{\pm c})\leq\frac{\sigma^2(f)c^2}{2(1-\alpha \|(f-\mu^*[f])^\pm\|_\infty c)}.
\end{align}
\end{lemma}
\begin{proof}
The cases where $\sigma^2(f)=0$ or one of $\|(f-\mu^*[f])^\pm\|_\infty=0$ are trivial, so suppose not. Using the self-adjoint functional calculus, one can see that $L$ inverts $A$ on $D(A)\cap L^2_0(\mu^*)$ and
\begin{align}
\left|\int fgd\mu^*\right|\leq \left(\int -A[g]gd\mu^*\right)^{1/2}\left(\int -L[f]fd\mu^*\right)^{1/2}
\end{align}
for all real-valued $f\in L^2_0(\mu^*)$, $g\in D(A,\mathbb{R})$.  

Hence, for any $g\in D(A,\mathbb{R})$ with $\|g\|_{L^2(\mu^*)}=1$ and any bounded, measurable $V$ (not necessarily related to $f$ at this point):
\begin{align}\label{integral_split}
&\int Vg^2d\mu^*\!=\!\!\int\! V(g-\mu^*[g])^2d\mu^*\!+2\mu^*[g]\!\int\! (V-\mu^*[V])gd\mu^*\!+\mu^*[V]\mu^*[g]^2\notag\\
\leq&\|V^+\|_\infty\Var_{\mu^*}[g]+\sqrt{2\sigma^2(V)}\sqrt{\langle -A[g],g\rangle}+\mu^*[V].
\end{align}
Using the Poincar{\'e} inequality and solving for $\langle -A[g],g\rangle$ gives
\begin{align}\label{Poincare_intermediate1}
&\langle -A[g],g\rangle\geq h\left(\int (V-\mu^*[V])g^2d\mu^*\right),\\
&h(r)\equiv 1_{r\geq 0}\frac{\sigma^2(V)}{2(M^\pm)^2}\left(\left(1+\frac{2M^+  }{\sigma^2(V)}r\right)^{1/2}-1\right)^2,\,\,\, M^+\equiv \alpha \|V^+\|_\infty.\notag
\end{align}
Letting $V=V_{\pm 1}=\pm(f-\mu^*[f])$ in \req{Poincare_intermediate1} and using the result to bound $\kappa$, \req{kappa_def2}, results in
\begin{align}
\kappa(V_{\pm c})\leq \sup_{r\in\mathbb{R}}\{cr-h(r)\}.
\end{align}
\req{Bernstein_bound_symmetric} then follows from solving the optimization problem.
\qed
\end{proof}

As with Theorem \ref{thm:poincare_UQ}, the Bernstein-type bound, \req{Bernstein_bound_symmetric}, implies a UQ bound:
\begin{theorem}\label{thm:poincare_UQ_reversible}
Under Assumption \ref{general_process_assump}, if the generator satisfies the Poincar\'e inequality \req{Poincare_def} and is self-adjoint on $L^2(\mu^*)$ then  for any bounded measurable $f:\mathcal{X}\to\mathbb{R}$ the bounds \req{standard_goal_oriented_bound} and \req{Bernstein_inf_bound} hold with
\begin{align}
M^\pm=\alpha \|(f-\mu^*[f])^\pm\|_\infty,\,\,\,\sigma^2=\sigma^2(f),\,\,\, \eta=\frac{1}{T}R(\widetilde P^{\widetilde{\mu}}_T||P^{\mu^*}_T).
\end{align}
\end{theorem}

Other variations can be derived using a Liapunov function.  First we need a couple of definitions, taken from Section 4 of \cite{doi:10.1137/S0040585X97986667}. Also, see this reference for further Liapunov function results that could likely be adapted to produce UQ bounds.
\begin{definition}
A measurable function $G:\mathcal{X}\to\mathbb{R}$ is in the $\mu^*$-extended domain of the generator, $D_{e,\mu^*}(A)$, if there is some measurable $g:\mathcal{X}\to\mathbb{R}$ such that $\int_0^t|g|(X_s)ds<\infty$ $P^{\mu^*}$-a.s. and one $P^{\mu^*}$-version of
\begin{align}
M_t(G)\equiv G(X_t)-G(X_0)-\int_0^t g(X_s)ds
\end{align}
is a local $P^{\mu^*}$-martingale.

$U\in D_{e,\mu^*}(A)$ is called a Liapunov function if $U\geq 1$ and there exist a measurable $\phi:\mathcal{X}\to(0,\infty)$ and $b>0$ such that
\begin{align}\label{liap_def}
-\frac{A[U]}{U}\geq \phi-b\,\,\,\mu^*\text{-a.s.}
\end{align}
\end{definition}

As shown in \cite{doi:10.1137/S0040585X97986667}, given a Liapunov function one can derive a bound on $\kappa(V_{\pm c})$; our method then produces a corresponding UQ bound:
\begin{theorem}\label{thm:Liap}
In addition to  Assumption \ref{general_process_assump}, assume the generator, $A$, is self-adjoint on $L^2(\mu^*)$ and satisfies the Poincar\'e inequality \req{Poincare_def}, and that we have a Liapunov function $U$ with $-A[U]/U\geq \phi-b$.

Given an observable $f\in L^2(\mu^*,\mathbb{R})$ with $\|(f-\mu^*[f])^\pm/\phi\|_\infty<\infty$, we have  the UQ bounds \req{standard_goal_oriented_bound} and \req{Bernstein_inf_bound}, where
\begin{align}
M^\pm=(1+\alpha b)\|(f-\mu^*[f])^\pm/\phi\|_\infty,\,\,\,\sigma^2=\sigma^2(f),\,\,\, \eta=\frac{1}{T}R(\widetilde P^{\widetilde{\mu}}_T||P^{\mu^*}_T).
\end{align}

\end{theorem}
\begin{proof}
First let $V$ be a bounded measurable function. This part of the proof proceeds similarly to that of Lemma \ref{lemma:Bernstein_bound_symmetric}, but rather than taking the supremum of $V^+$ in \req{integral_split}, one instead uses \req{liap_def} to compute the following bound, where $g\in D(A,\mathbb{R})$ with $\|g\|_{L^2(\mu^*)}=1$:
\begin{align}
\int Vg^2d\mu^*\leq &\mu^*[V]+\sqrt{2\sigma^2(V)}\sqrt{\langle -A[g],g\rangle}\\
&+\|V^+/\phi\|_\infty\int \left(-\frac{A[U]}{U}+b\right)(g-\mu^*[g])^2d\mu^*.\notag
\end{align}

Next, use the bound found in Lemma 5.6 in \cite{Guillin2009},
\begin{align}\label{liap_int_bound}
\int -\frac{A[U]}{U}(g-\mu^*[g])^2d\mu^*\leq  \langle-A[g],g\rangle,
\end{align}
 and  proceed as in Lemma \ref{lemma:Bernstein_bound_symmetric} to obtain
\begin{align}
\kappa(\pm c V)\leq \pm c\mu^*[V]+\frac{\sigma^2(V)c^2}{2(1-(1+\alpha b)\|V^\pm /\phi\|_\infty c)}
\end{align}
for all $0<c<1/((1+\alpha b)\|V^\pm /\phi\|_\infty)$.  If $f$ is bounded then applying this to $V=f-\mu^*[f]$ and using Corollary \ref{thm:master_UQ_bound} and Lemma \ref{lemma:Bernstein} gives the claimed UQ bound.
 
 For general $f\in L^2(\mu^*,\mathbb{R})$ with $\|(f-\mu^*[f])^\pm/\phi\|_\infty<\infty$, we employ a similar method to Corollary 3  in \cite{WU2000435}: Define $V=f-\mu^*[f]$ and $V^n=V1_{|V|<n}$ (not to be confused with the $n$th power of $V$).  Applying the above result to $V^n$ and then using Fatou's lemma and $L^2$-continuity of the asymptotic variance gives
 \begin{align}
  \frac{1}{T}\Lambda_{P^{\mu^*}_T}^{\widehat{f}_T}(\pm c)\leq&\frac{1}{T}\log\left(\|\mathcal{P}^{V_{\pm c}}_T[1]\|\right)\leq \liminf_{n\to\infty}\frac{1}{T}\log\left(\|\mathcal{P}^{\pm c V^n}_T[1]\|\right)\\
 \leq& \liminf_{n\to\infty} \left(\pm c\mu^*[ V^n]+\frac{\sigma^2(V^n)c^2}{2(1-(1+\alpha b)\|(f-\mu^*[f])^\pm/\phi\|_\infty c)}\right)\notag\\
 =&\frac{\sigma^2(f)c^2}{2(1-(1+\alpha b)\|(f-\mu^*[f])^\pm/\phi\|_\infty c)}.\notag
 \end{align}
 Having extended the bound on the cumulant generating function to such $f$, the claimed UQ bound follows from Proposition \ref{thm:goal_div}.
 \qed
\end{proof}

\subsection{Poincar{\'e} Inequality Examples}\label{sec:Poincare_ex}
The study of Poincar\'e inequalities has a long history which we do not attempt to recount here.  For a detailed discussion, see \cite{wang2006functional}, which covers Poincar{\'e} inequalities for both continuous-time Markov chains and diffusions.  Criteria for diffusions can also be found, for example, in \cite{bakry2008,BAKRY2008727}.  

The following example illustrates that the Bernstein-type bounds used in this paper can be sharp for Markov processes.
\subsubsection{A simple Liapunov example:  the $M/M/\infty$ queue.}\label{ex:queue}   Following \cite{doi:10.1137/S0040585X97986667},  let us consider the (simple) example of an $M/M/\infty$ queuing system which has infinitely many servers, each   with a service rate $\rho$ and   an arrival rate $\lambda$.  The state space is $\mathbb{N}$ and the generator is given by 
\begin{equation}\label{eq:liapmm}  
A[f](n) = \lambda f(n+1) -( \lambda +\rho n) f(n) +  \rho n f(n-1).
\end{equation} 
The invariant measure $\mu^*$ is a Poisson distribution with parameter $\lambda/\rho$.  An explicit computation shows (see e.g. \cite{Chafai}) 
that $\Var_{\mu^*}[ \mathcal{P}_t f]  \le e^{-2\rho t }  \Var_{\mu^*}[f]$ and thus the Poincar\'e constant is $1/\rho$. 

To construct a Liapunov function take $U(n) = \kappa^n$ with $\kappa >1$; we then  have 
\begin{equation}
- \frac{A[U]}{U}(n) = \rho n ( 1 - \kappa^{-1}) - \lambda (\kappa -1) \,, 
\end{equation}
and we can apply Theorem \ref{thm:Liap} to any function $f$ with $|f|\le C(n +\delta)$ for some $\delta >0$.   

It is instructive to consider further the case of the mean number of customers in the queue, i.e., $f = n$ and  $\widehat{f} = f - \mu^*[f] = n - \lambda/\rho$. 
From \req{eq:liapmm} we obtain 
\begin{equation}
( A + \rho (1 - \kappa^{-1}) \widehat{f})[U](n) = \lambda\frac{(\kappa-1)^2}{\kappa}U(n)  
\end{equation}
and thus $U$ is an eigenvector for $ A + \rho(1-\kappa^{-1})\widehat{f}$ with eigenvalue $\lambda\frac{(\kappa-1)^2}{\kappa}$. 
By the Perron-Frobenius theorem and Rayleigh's principle we obtain that 
\begin{equation}\label{eq:queue_Lambda_def}
\Lambda(c) \equiv \lim_{T\to \infty} T^{-1}\Lambda_{P_T^{\mu^*}}^{\widehat{f}_T}( c)
\end{equation}
is the maximal eigenvalue of $A + c \widehat{f}$ and thus $\Lambda\left( \rho(1-\kappa^{-1})\right) = \lambda\frac{(\kappa-1)^2}{\kappa}$ or 
equivalently 
$ \Lambda(c) = \frac{\lambda c^2}{\rho^2( 1 - c \rho^{-1})}$. 
Since $A \widehat{f}(n) = \lambda - \rho n$ we can solve the Poisson equation:   $(-A)^{-1} \widehat{f} = \widehat{f}/\rho$ and thus the asymptotic variance is $\sigma^2(f) = 2 \langle (-A)^{-1} \widehat{f}\,,\, \widehat{f}\rangle = 2\rho  ^{-1}\Var_{\mu^*} [f] = 2\lambda/\rho^2$. 
As a consequence we have 
\begin{equation}\label{eq:queue_Lambda}
\Lambda(c) = \frac{\sigma^2(f)c^2}{2( 1 - c \rho^{-1})},
\end{equation} 
which shows that Bernstein bounds can be sharp in the context  of Markov processes, contrary to the IID setting.

\subsubsection{Poincar{\'e} Inequality from Exponential Convergence} 
It is well-known that, when the generator, $A$, is self-adjoint, a Poincar{\'e} inequality is equivalent to exponential convergence in the $L^2(\mu^*)$-norm. Here, we discuss a method for deriving a Poincar{\'e} inequality from exponential convergence in alternative norms.

 First, note that one only needs exponential $L^2$-convergence  on a subset with dense span to conclude a Poincar{\'e} inequality (see Lemma 1.2 in \cite{cattiaux:hal-00461085}):
\begin{lemma}\label{sa_poincare_lemma}
Suppose $(A,D(A))$ is self-adjoint, $F\subset L^2(\mu^*)$ has dense span,   and there exists  $\alpha >0$ such that the following holds:\\
For every $f\in F$ there exists $C_f\geq 0$ such that
\begin{align}\label{exp_conv}
\|\mathcal{P}_{t}[f]-\mu^*[f]\|_2\leq C_fe^{- t/\alpha}  \text{ for all $t\geq 0$.}
\end{align}
Then a  Poincar{\'e} inequality, \req{Poincare_def}, holds with constant $\alpha$.
\end{lemma}

The following result shows how to obtain a Poincar{\'e} inequality (with an explicit constant) from exponential convergence in a pair of weighted norms.  
\begin{theorem}\label{thm:poinc_exp_conv}
Suppose $(A,D(A))$ is self-adjoint, and $W:\mathcal{X}\to[1,\infty)$ is measurable. Define the following norms on measurable functions $\phi:\mathcal{X}\to \mathbb{R}$ and  signed measures $\pi$ on $\mathcal{X}$:
\begin{align}\label{weighted_norms}
|\phi|_{W}=\sup_{x\in\mathcal{X}}\frac{|\phi(x)|}{W(x)},\,\,\, |\pi|_{W}=\int Wd|\pi|.
\end{align}

  Suppose we have  $\lambda\geq 0, \rho\geq 0$ with at least one nonzero, and that for every bounded measurable $h:\mathcal{X}\to [0,\infty)$ with $\int hd\mu=1$ there exist  $C_h, D_h\in[0,\infty)$ such that for all $t\geq 0$:
\begin{align}\label{weighted_conv1}
|\mathcal{P}_{t}[h]-1|_{W}\leq D_h e^{-\rho t}
\end{align}
and the measure $d\nu=hd\mu^*$ satisfies
\begin{align}\label{weighted_conv2}
|\mathcal{P}^\dagger_{t}[\nu]-\mu^*|_{W}\leq C_h e^{-\lambda t},
\end{align}
where $\mathcal{P}^\dagger_t$ denotes the action of the semigroup $p_t$ on measures.

Then $A$ satisfies the  Poincar{\'e} inequality
\begin{align}
\Var_{\mu^*}[g]\leq -\frac{2}{\lambda+\rho}  \langle A[g],g\rangle\,\,\text{ for all $g\in D(A,\mathbb{R})$.}
\end{align}
\end{theorem}
\begin{proof}
The proof is similar to that of Theorem 2.1 in \cite{BAKRY2008727}.  The key is to take $h$ as above, let $d\nu=hd\mu^*$,  use symmetry of $\mathcal{P}_t$ to compute
\begin{align}
\|\mathcal{P}_t[h]-1\|_2^2=&\int \frac{|\mathcal{P}_t[h]-1|}{W} W|\mathcal{P}_t[h]-1|d\mu^*\leq|\mathcal{P}_t[h]-1|_W|\mathcal{P}_t^\dagger[\nu]-\mu^*|_W,
\end{align}
and then apply Lemma \ref{sa_poincare_lemma}.
\qed
\end{proof}
Exponential convergence in  norms  of the form $|\cdot|_{W}$  can be obtained from the existence of an appropriate Liapunov function (see \cite{HarierNotes,HairerMattingly2011}), making Theorem \ref{thm:poinc_exp_conv} a practical method for obtaining Poincar{\'e} inequalities.

\begin{remark}
The proof of Lemma \ref{sa_poincare_lemma} can be generalized to only require \req{exp_conv} to hold along a sequence $t_n^f$ converging to $\infty$.  Hence, Theorem \ref{thm:poinc_exp_conv} can also be generalized to only require \req{weighted_conv1} and \req{weighted_conv2} along a common sequence $t_n^h\to\infty$. 
\end{remark}

\subsection{$\log$-Sobolev Inequalities}\label{sec:log_sobolev}
Next consider the $\log$-Sobolev inequality with constant $\beta>0$:
\begin{align}\label{log_Sob_def}
\int g^2 \log(g^2)d\mu^{*}\leq -\beta\int A[g]gd\mu^{*}
\end{align}
for all  $g\in D(A,\mathbb{R})$ with $\|g\|_{L^2(\mu^{*})}=1$.

We will employ the following generalization of the Feynman-Kac semigroup for (possibly) unbounded potentials.  The subsequent proposition was shown in Corollary 4 in \cite{WU2000435}.  For completeness purposes, we outline  the proof. 
\begin{proposition}\label{Kac_log_sobolev}
Let  $A$ be the generator of $\mathcal{P}_t$ and $\mu^*$ be an invariant measure for the adjoint semigroup, $\beta>0$, and  assume the $\log$-Sobolev inequality, \req{log_Sob_def}, holds for $\mu^*$ with constant $\beta$.

Finally, suppose that $V\in L^1(\mu^*)$ with $\int e^{\beta V}d\mu^*<\infty$. Then   $\mathcal{P}_t^V:L^2(\mu^*)\to L^2(\mu^*)$, defined by
\begin{align}
\mathcal{P}_t^V[g](x)=E^x\left[g(X_t)\exp\left(\int_0^t V(X_s)ds\right)\right],
\end{align}
are well-defined linear operators and the operator norm satisfies the bound
\begin{align}\label{Kac_bound2}
\|\mathcal{P}_t^V\|\leq \left(\int e^{\beta V}d\mu^*\right)^{ t/\beta}.
\end{align}

\end{proposition}
\begin{proof}
First assume $V$ is bounded.  \req{kac_norm_bound} gives $\|\mathcal{P}_t^V\|\leq e^{t\kappa(V)}$. Applying the $\log$-Sobolev inequality together with  the Gibbs variational principle, \req{gibbs}, we obtain
\begin{align}
\kappa(V)\leq&\beta^{-1}\sup\left\{-\int g^2\log(g^2)d\mu^*+\int \beta V|g|^2d\mu^*:\|g\|_{L^2(\mu^*)}=1\right\}\\
=& \beta^{-1}\!\!\!\!\!\!\!\!\!\!\!\!\sup_{d\nu=g^2d\mu^*:\|g\|_2=1} \!\!\!\!\!\!\!\!\!\!\{E_{\nu}[\beta V]-R(\nu||\mu^*)\}=\beta^{-1}\log\left( \int \exp\left(\beta V\right) d\mu^*\right),\notag
\end{align}
which proves the claim.

The case of unbounded $V$ satisfying the assumptions of the theorem is obtained by letting $V^n=V1_{|V|\leq n}$,  and then using Fatou's lemma, the result for bounded $V$, and  dominated convergence to compute 
\begin{align}
\|\mathcal{P}^V_t\|\leq& \liminf_{n\to\infty}\|\mathcal{P}^{V^n}_t\|\leq \liminf_{n\to\infty}\left(\int e^{\beta V^n}d\mu^*\right)^{ t/\beta}=\left(\int e^{\beta V}d\mu^*\right)^{ t/\beta}.\notag
\end{align}
\qed
\end{proof}

Using Proposition \ref{Kac_log_sobolev}, a UQ bound of the form \req{standard_goal_oriented_bound} can be derived that covers a class of unbounded observables:
\begin{theorem}\label{thm:sobolev_UQ}
In addition to  Assumption \ref{general_process_assump}, assume the $\log$-Sobolev inequality, \req{log_Sob_def}, holds and we have an observable $f\in L^1(\mu^*,\mathbb{R})$ and $c_-<0<c_+$ such that for all $c\in(c_-,c_+)$:
\begin{align}\label{sobolev_int_condition}
\int \exp\left(\beta  V_{c}\right)d\mu^*<\infty.
\end{align}

Then a UQ bound of the form \req{standard_goal_oriented_bound} holds with
\begin{align}\label{Sobolev_lambda}
\Lambda(c)=\left\{ \begin{array}{cl} \frac{1}{\beta}\log\left(\int e^{\beta V_{ c}}d\mu^*\right) & \textrm{ if } c\in(c_-,c_+)  \\  + \infty & \textrm{ otherwise}
\end{array} \,.
\right. 
\end{align}

In addition,  the asymptotic result  \req{asymp_bound} holds with
\begin{align}
\Lambda^{\prime\prime}(0)=\beta \Var_{\mu^*}[f],\,\,\, \eta=\frac{1}{T}R(\widetilde P^{\widetilde{\mu}}_T||P^{\mu^*}_T).
\end{align}
\end{theorem}
\begin{proof}
The bound \req{Kac_bound2} implies  $E^{\mu^*}[\exp( cf_T)]<\infty$ for  $c\in(c_-,c_+)$, hence $f_T\in\mathcal{E}(P^{\mu^*}_T)$ and  the Gibbs information inequality, \req{goal_oriented_bound}, applies. As in \req{Lambda_bound}, the cumulant generating function can be bounded using the Feynman-Kac semigroup bound, \req{Kac_bound2}. Combining this with \req{goal_oriented_bound} yields  a bound of the form \req{standard_goal_oriented_bound}, with $\Lambda$ as defined in \req{Sobolev_lambda}.
\qed
\end{proof}
The ideas in this section can be extended to $F$-Sobolev inequalities; see Appendix \ref{app:F_sobolev}.

\subsubsection{Example: Diffusions}\label{sec:log_sob_diffusion}
Let  $V$ be a $C^2$ potential, bounded below, and growing sufficiently fast at infinity. Consider the diffusion with  generator $A=\Delta - \nabla V\cdot \nabla$ and invariant measure $\mu^*(dx)=e^{-V(x)}dx$.   First, it is useful to note that a $\log$-Sobolev inequality with constant $\beta$ implies a Poincar{\'e} inequality with constant $\alpha=\beta/2$ \cite{rothaus1978}. In \cite{CarlenLoss}, the following sufficient condition for a $\log$-Sobolev inequality was obtained:

Suppose $A$ satisfies a Poincar{\'e} inequality  with constant $\alpha$ (references on Poincar{\'e} inequalities can be found in Section \ref{sec:Poincare_ex}) and that
\begin{align}
-C\equiv\inf_x\left\{\frac{1}{4}|\nabla V(x)|^2-\frac{1}{2}\Delta V(x)-\pi e^2V(x)\right\}>-\infty.
\end{align}
Then $A$ satisfies a $\log$-Sobolev inequality with constant
\begin{align}\label{eq:beta_log_sobolev}
\beta=3\alpha+\frac{1}{(1+\alpha |C|)\pi e^2}.
\end{align}

As a second example, if the Hessian of $V$ is bounded below,
\begin{align}
D^2V(x)\geq 2 \beta^{-1} I,
\end{align}
for some $\beta>0$ (unrelated to the $\beta$ in \req{eq:beta_log_sobolev}) then a $\log$-Sobolev inequality holds with constant $\beta$ \cite{BakryEmery}.  A UQ bound corresponding to the associated Poincar{\'e} inequality with constant $\alpha\equiv\beta/2$ was given in the introduction in \req{hessian_UQ_bound}.  

\section{Functional Inequalities and UQ for Discrete-Time Markov Processes} \label{sec:discrete}
In this section we show how the above framework can be applied to obtain UQ bounds for invariant measures of discrete-time Markov processes.

  Again, let $\mathcal{X}$ be a Polish space, and suppose we have  one-step transition probabilities $p(x,dy)$ and $\widetilde p(x,dy)$ on $\mathcal{X}$ with invariant measures $\mu^*$ and $\widetilde\mu^*$ respectively.  Assume that $R(\widetilde\mu^*||\mu^*)<\infty$.
  
    Define the bounded linear operator $\mathcal{P}$ on $L^2(\mu^*)$,
  \begin{align}
\mathcal{P}[f](x)\equiv \int f(y)p(x,dy),
\end{align} 
and similarly for $\widetilde{\mathcal{P}}$ on $L^2(\widetilde\mu^*)$.

 We obtain UQ bounds for expectations in $\mu^*$ and $\widetilde\mu^*$ by constructing continuous-time processes with these same invariant distributions. Specifically, in  Appendix  \ref{app:jump} (see Theorem \ref{theorem:markov_family}) we obtain  c\`adl\`ag Markov families
$(\Omega,\mathcal{F},\mathcal{F}_t,X_t,\{P^x\}_{x\in\mathcal{X}})$ and $(\Omega,\mathcal{F},\mathcal{F}_t,X_t,\{\widetilde P^x\}_{x\in\mathcal{X}})$, whose transition probabilities $p_t$ and $\widetilde p_t$, respectively, (not to be confused with $p$ and $\widetilde p$) satisfy the following:
\begin{enumerate}
\item $\mu^*$ is invariant for $p_t$ for all $t\geq 0$, and similarly for $\widetilde \mu^*$ and $\widetilde p_t$  (see Theorem \ref{thm:markov}).
\item The continuous-time semigroup, $\mathcal{P}_t$, on $L^2(\mu^*)$ constructed from $p_t$ is
\begin{align}\label{T_bounded_gen}
\mathcal{P}_t=\exp(t(\mathcal{P}-I)).
\end{align}
Specifically, $\mathcal{P}_t$ has bounded generator $A=\mathcal{P}-I$ (see Theorem \ref{thm:markov}).  Note that we will also refer to $A$ as the generator of the discrete-time Markov process.
\item The relative entropy  rate of the continuous-time process can be bounded by the relative entropy of the discrete-time process as follows:
\begin{align}\label{discrete_rel_ent}
R(\widetilde P^{\widetilde\mu^*}_T||P^{\mu^*}_T)\leq R(\widetilde\mu^*||\mu^*)+T\int R(\widetilde p(x,\cdot)|| p(x,\cdot))\widetilde\mu^*(dx)
\end{align}
for all $T>0$ (see Theorem \ref{thm:rel_ent} and Corollary \ref{corollary:rel_ent_bound}).
\end{enumerate}
\begin{remark}
While the above construction, and the computation of the relative entropy, is standard for countable state spaces (see the discussion in Section \ref{Example:CTMC}), for our purposes it is necessary to work with general state spaces; to the best of our knowledge, the relative entropy bound \req{discrete_rel_ent} is new in this case.  

General state spaces are of interest, for example, when one is working with Markov chain Monte Carlo samplers, $p(x,dy)$ and $\widetilde p(x,dy)$, for  measures, $\mu^*$ and $\widetilde{\mu}^*$ respectively, on $\mathbb{R}^n$. In this setting, to use our UQ method, one can construct the ancillary continuous-time Markov chain on $\mathbb{R}^n$, as outlined in Appendix \ref{app:jump}, and then apply the relative entropy bound \req{discrete_rel_ent}.
\end{remark}

The Markov families $P^x$ and $\widetilde P^x$, obtained via the above construction, satisfy Assumption \ref{general_process_assump}.  Hence, if the generator $\mathcal{P}-I$ satisfies any of the functional inequalities covered  in Section \ref{sec:Kac} then the general results therein imply UQ bounds for expectations in the invariant measures $\mu^*$ and $\widetilde\mu^*$, with \req{discrete_rel_ent} providing a bound on the relative entropy rate.
\begin{remark}
  Note that here, we must take $\widetilde\mu=\widetilde\mu^*$ for the bounds to apply to the original discrete-time process, otherwise one  obtains UQ bounds for ergodic averages of $f(X_t)$ under the auxiliary continuous-time Markov family.  
\end{remark}

 For example, a Poincar{\'e} inequality for the generator $\mathcal{P}-I$,
\begin{align}\label{discrete_poincare1}
\Real(\langle (\mathcal{P}-I)g,g\rangle)\leq -\alpha^{-1}\|P^\perp g\|_{L^2(\mu^*)}^2,\,\,\, g\in L^2(\mu^*),\,\,\alpha>0,
\end{align}
implies that for any bounded measurable $f:\mathcal{X}\to\mathbb{R}$, we have
\begin{align}\label{discrete_poincare_UQ}
&\pm\left( \widetilde\mu^*[f] -\mu^*[f]\right)\leq \sqrt{2\sigma^2\eta}+M^\pm\eta,\\
&\sigma^2=2\alpha\Var_{\mu^*}[f],\,\,\,M^\pm=\alpha \|(f-\mu^*[f])^\pm\|_\infty,\notag\\
 &\eta=\int R(\widetilde p(x,\cdot)|| p(x,\cdot))\widetilde\mu^*(dx).\notag
\end{align}
This follows from  Theorem \ref{thm:poincare_UQ}, after taking $T\to\infty$ (recall the assumption $R(\widetilde\mu^*||\mu^*)<\infty$).

We illustrate these discrete-time UQ bounds with a pair of examples:
\subsection{Example: Random Walk on a Hypercube}
Consider the symmetric random walk on the $d$-dimensional hypercube $\mathcal{X}=\{-1,1\}^d$ i.e. the transition probabilities are defined by uniformly randomly selecting a coordinate, $i\in\{1,...,d\}$, and then independently and uniformly selecting the sign, $1$ or $-1$, with which to update the selected component.

The uniform measure, $\mu^*$, on $\mathcal{X}$ is invariant and the  process is reversible on $(\mathcal{X},\mu^*)$. The eigenvalues and eigenvectors of the transition matrix can be found explicitly; see Example 12.15 in \cite{levin2017markov}.  In particular, the second largest eigenvalue is $\lambda_2=1-1/d$, hence we obtain the following Poincar{\'e} inequality:
\begin{align}
\Real(\langle (\mathcal{P}-I)g,g\rangle)\leq -\frac{1}{d}\|P^\perp g\|^2_{L^2(\mu^*)},\,\,\, g\in L^2(\mu^*).
\end{align}
Assuming $R(\widetilde\mu^*||\mu^*)<\infty$, we then obtain the UQ bound \req{discrete_poincare_UQ} with $\alpha=d$.

\subsection{Example: Exclusion Chain}
Derivation of functional inequalities for many discrete-time Markov processes can be found in \cite{diaconis1996}.  Here we investigate the resulting UQ bounds for one of these examples; see Section 4.6 in  the above reference and also \cite{Diaconis1993} for further details and proofs regarding this example.

Let $(V,E)$ be a  symmetric, connected graph with $n$ vertices.  Let $d(x)$ be the degree of a vertex $x\in V$ and $d_0=\max_x d(x)$.  Fix $r\leq n$.  The $r$-exclusion process is a Markov chain with state space being the set of cardinality $r$ subsets of $V$.  Informally stated, the transition probabilities are defined as follows:  Given an $r$-subset $A$ (i.e., state of the chain), pick an element $x\in A$ with probability proportional to its degree.  Uniformly randomly pick a vertex $y$ out of all those connected with $x$.  If $y$ is not in $A$ then transition to the new state  $(A\setminus\{x\})\cup\{y\}$.  Otherwise, the chain remains at the set $A$. 

For each $(x,y)\in V\times V$, fix a path $\gamma_{x,y}$ from $x$ to $y$ in the graph and let $|\gamma_{x,y}|$ be its length. Define
\begin{align}
\Delta_0=&\max_{e_0\in V}\left\{\sum_{(x,y):e_0\in\gamma_{x,y}}|\gamma_{x,y}|\right\},\,\,\,\,
d_r=\max_{A\subset V:|A|=r}\left\{\frac{1}{r}\sum_{a\in A}d(a)\right\}.
\end{align}
The generator of this  Markov chain satisfies both a Poincar{\'e} inequality and a $\log$-Sobolev inequality with respective constants being
\begin{align}\label{exclusion_const}
\alpha=rd_r\Delta_0/n,\,\,\,\,\,\beta=3rd_r\Delta_0\log(n)/n.
\end{align}

Then, assuming $R(\widetilde\mu^*||\mu^*)<\infty$, the above Poincar{\'e} inequality implies the UQ bound \req{discrete_poincare_UQ} with $\alpha$ as in \req{exclusion_const}, and the $\log$-Sobolev inequality results in
\begin{align}
\pm\left(\widetilde{\mu}^*[f]-\mu^*[f]\right)\leq \inf_{c>0}\left\{\frac{1}{c\beta}\log\left(\int \exp\left(\pm\beta c(f-\mu^*[f])\right) d\mu^*\right)+\frac{\eta}{c}\right\},
\end{align}
with $\beta$ and $\eta$ as in \req{exclusion_const} and \req{discrete_poincare_UQ} respectively.

\section{Bounding the Relative Entropy Rate}\label{sec:rel_ent}
For any $\eta>0$, the results derived in the previous sections provide  UQ bounds over the class of all alternative models that satisfy a relative entropy bound of the form 
\begin{align}
H_T(\widetilde P^{\widetilde \mu}||P^{\mu^*})\equiv\frac{1}{T}R(\widetilde P^{\widetilde \mu}_T||P^{\mu^*}_T)\leq \eta.
\end{align}
  In this section, we study in more detail the dependence of $H_T$ on $T$ and on the models $\widetilde{P}^{\widetilde \mu}$ and $P^{\mu^*}$.  Specifically, we derive upper bounds on $H_T$ in various settings that can be substituted for $H_T$ in the general UQ bound \req{standard_goal_oriented_bound}. Here, it will make little difference whether the initial distribution for the $P$-process is invariant or not, so we no longer make that assumption when deriving the relative entropy bounds;  $\mu$ will denote an arbitrary initial distribution.

Deriving bounds on the relative entropy is a very application-specific problem.  We will cover several examples in detail: continuous-time Markov chains, semi-Markov processes, change of drift in SDEs, and numerical methods for SDEs with additive noise.

\subsection{Example: Continuous-Time Markov Chains}\label{Example:CTMC}

Let $\mathcal{X}$ be a countable set, $P^{\mu}$, $\widetilde P^{\widetilde \mu}$ be probability measures on $(\Omega,\mathcal{F})$ and $X_t:\Omega\to\mathcal{X}$ such that  $P^{\mu}$ (resp. $\widetilde P^{\widetilde{\mu}}$) makes $(\Omega,\mathcal{F},X_t)$ a continuous-time Markov chain (CTMC) with transition probabilities $a(x,y)$ (resp. $\widetilde a(x,y)$), jump rates $\lambda(x)$ (resp. $\widetilde\lambda(x)$), and initial distribution $\mu$ (resp. $\widetilde\mu$).  Let $\mathcal{F}_t$ be the natural filtration for $X_t$ and $X^J_n$ be the embedded jump chain with jump times $J_n$.

Suppose $\widetilde{\mu}\ll \mu$, $\lambda$ and $\widetilde{\lambda}$ are positive and bounded above, and for all $x,y\in \mathcal{X}$ we have $a(x,y)=0$ iff $\widetilde{a}(x,y)=0$.  Then for any $T>0$ we have $\widetilde P^{\widetilde{\mu}}|_{\mathcal{F}_T}\ll P^\mu|_{\mathcal{F}_T}$ and
\begin{align}\label{eq:rel_ent_CTMC}
&R(\widetilde P^{\widetilde{\mu}} |_{\mathcal{F}_T}||P^\mu |_{\mathcal{F}_T})\\
=&R(\widetilde\mu||\mu)+\widetilde{E}^{\widetilde\mu}\left[\int_0^T\widetilde{F}(X_s)\widetilde\lambda (X_s)ds\right] -\widetilde{E}^{\widetilde{\mu}}\left[\int_0^T \widetilde\lambda (X_s)-\lambda (X_s)ds \right],\notag\\
&\widetilde F(x)\equiv\sum_{z\in\mathcal{X}}\widetilde a (x,z)\log\left(\frac{\widetilde\lambda (x)\widetilde a (x,z)}{\lambda (x)a (x,z)}\right).\notag
\end{align}

To simplify further, if $\widetilde\mu=\widetilde\mu^*$ is an invariant measure then
\begin{align}
&R(\widetilde P^{\widetilde{\mu}^*} |_{\mathcal{F}_T}||P^{\mu} |_{\mathcal{F}_T})=R(\widetilde\mu^*||\mu)\\
&+T\left(\sum_{x\in\mathcal{X}}\sum_{z\in\mathcal{X}}\widetilde\mu^*(x)\widetilde\lambda (x)\widetilde a (x,z)\log\left(\frac{\widetilde\lambda (x)\widetilde a (x,z)}{\lambda (x)a (x,z)}\right)-\sum_{x\in\mathcal{X}}\widetilde\mu^*(x)\left(\widetilde\lambda (x)-\lambda (x)\right) \right).\notag
\end{align}
See the supplementary materials to \cite{DKPP} and Proposition 2.6 in Appendix 1 of \cite{kipnis2013scaling} for  details regarding these results.

\subsection{Example: Semi-Markov Processes}\label{Example:semi-Markov}
As we have noted previously, our results require $(X_t,P^x)$ to be Markov, but do {\em not} require the alternative model $(X_t,\widetilde{P}^x)$ to be Markov.  Here we discuss  one such class of examples, that of a semi-Markov perturbation of a continuous-time Markov chain.

Semi-Markov processes are continuous-time jump processes with memory (i.e., with nonexponential waiting times).  Such a process is defined by a jump chain, $X^J_n$,  jump times, $J_n$, and waiting times (i.e. jump intervals), $\Delta_{n+1}\equiv  J_{n+1}- J_n$, that satisfy
\begin{flalign}
&\widetilde{P}^{\widetilde{\mu}}( X^J_{n+1}=y, \Delta_{n+1}\leq t| X^J_1,..., X^J_{n-1}, X^J_n, J_1,..., J_n)\notag\\
&= \widetilde{P}^{\widetilde{\mu}}( X^J_{n+1}=y, \Delta_{n+1}\leq t| X^J_n)\equiv \widetilde{Q}_{X^J_n,y}(t).\notag
\end{flalign}
$\widetilde{Q}_{x,y}(t)$ is called the semi-Markov kernel; see, for example, \cite{janssen2006applied,limnios2012semi} for further details.  Note that a continuous-time Markov chain with embedded jump Markov-chain transition probabilities $a(x,y)$ and jump rates $\lambda(x)$ is  described by the semi-Markov kernel  
  \begin{align}\label{eq:base_semiMKer}
  Q_{x,y}(t)=a(x,y)\int_0^t \lambda(x)e^{-\lambda(x)s}ds.
  \end{align}  

A semi-Markov perturbation of \req{eq:base_semiMKer} with the same embedded jump Markov-chain but with modified (nonexponential) waiting times is described by a kernel of the form
 \begin{align}\label{eq:semiMarkov_wait_perturb}
 \widetilde{Q}_{x,y}(t)=a(x,y)\widetilde{H}_x(t).
 \end{align}
 \begin{remark}
Phase-type distributions constitute a useful semiparametric description of such alternative waiting-time distributions, going beyond the exponential case to describe systems with memory; see  \cite{doi:10.1002/asm.3150100403,bladt2017matrix} for details.
\end{remark} 
 
 The relative entropy rate,
 \begin{align}
\eta\equiv \limsup_{T\to\infty} \frac{1}{T} R( \widetilde{P}^{\widetilde{\mu}}|_{\mathcal{F}_T}||P^{\widetilde{\mu^*}}|_{\mathcal{F}_T}),
\end{align}
between semi-Markov processes was obtained in \cite{10.2307/3216060} under the appropriate ergodicity assumptions.  When the base process has the form \req{eq:base_semiMKer} and the alternative process has the form \req{eq:semiMarkov_wait_perturb}, the relative entropy rate can be expressed in terms of the relative entropy of the waiting-time distributions:
 \begin{align}\label{eq:semimarkov_rel_ent}
 \eta=&\frac{1}{\widetilde{m}_{ \pi}}\sum_x\pi(x)R(\widetilde{H}_x||H_x),\,\, \,\,\widetilde{m}_{{\pi}}\equiv\sum_{x}{\pi}(x)\int_0^\infty (1-\widetilde{H}_x(t))dt,
 \end{align}
 where $\pi$ is the invariant distribution for the Markov chain $a(x,y)$.
 \begin{remark}
The quantity $\widetilde{m}_\pi$ is the mean sojourn time under the invariant distribution, $\pi$, and $\sum_x\pi(x)R(\widetilde{H}_x||H_x)$ can be thought of as the mean relative entropy of a single jump (comparing the alternative and base model waiting-time distributions). The formula for $\eta$, \req{eq:semimarkov_rel_ent}, therefore has the  intuitive meaning of an information loss per unit time.
\end{remark}

\subsubsection{Semi-Markov Perturbations of a $M/M/\infty$-Queue}

As a concrete example, we consider semi-Markov perturbations of an $M/M/\infty$-queue with service rate $\rho$ and with an arrival rate $\lambda$.  The  embedded jump Markov-chain  is given by
\begin{align}\label{eq:jump_chain}
a(x,x+1)=\lambda/(\lambda+\rho x),\,\,\,\,a(x,x-1)=\rho x/(\lambda+\rho x)
\end{align}
and the waiting-times are exponentially distributed with jump rates 
 \begin{align}\label{eq:lambda_def}
 \lambda(x)=\alpha+\rho x.
 \end{align} 
 \req{eq:jump_chain}, has invariant distribution
  \begin{align}
  \pi(x)= \frac{(\alpha+\rho x)(\alpha/\rho)^x }{2\alpha x!} e^{-\alpha/\rho}.
\end{align}

 Taking $T\to \infty$ in \req{goal_oriented_bound3} and using \req{eq:queue_Lambda_def}, \req{eq:queue_Lambda}, and \req{eq:semimarkov_rel_ent} we therefore obtain the following asymptotic upper bound on the average queue length in the alternative model:
\begin{align}\label{eq:time_avg_UQ_limit}
&\limsup_{T\to\infty}\left( \widetilde E^{\widetilde{\mu}}\left[\frac{1}{T}\int_0^T X_t dt\right]-\alpha/\rho\right)\\
\leq& \inf_{0<c<\rho}\left\{\frac{1}{c}\frac{\alpha c^2}{\rho^2 (1-c/\rho)}+\frac{1}{c}\eta\right\}=\left(2\sqrt{ \eta/\alpha}+\eta/\alpha\right)\frac{\alpha}{\rho},\notag
\end{align}
where
 \begin{align}
 \eta=& \frac{1}{\widetilde{m}_{ \pi}}\sum_x\pi(x)R(\widetilde{H}_x||H_x),\\
 \widetilde{m}_{{\pi}}=&\sum_{x}{\pi}(x)\int_0^\infty (1-\widetilde{H}_x(t))dt, \,\,\,H_x(t)=\int_0^t \lambda(x)e^{-\lambda(x)s}ds.\notag
 \end{align}
 Note that the only ingredient from the alternative model that is needed in \req{eq:time_avg_UQ_limit} is   $\widetilde{H}_x$, and given this, the bounds are generally straightforward to evaluate.

\subsection{Example: Change of Drift for SDEs}
Next, consider the case where $P^x$ and $\widetilde P^x$ are the distributions on  $C([0,\infty),\mathbb{R}^n)$ of the solution flows $X_t^x$ and $\widetilde X_t^x$ of a pair of SDEs. More precisely:
\begin{assumption}\label{assump:SDE}
Assume:
\begin{enumerate}
\item $X^x_t$ and $\widetilde X^x_t$ are weak solutions to the $\mathbb{R}^n$-valued SDEs, on filtered probability spaces satisfying the usual conditions \cite{karatzas2014brownian}:
\begin{align}
&dX^x_t=b(X_t^x)dt+\sigma(X_t^x)dW_t,\,\, X^x_0=x,\\
& d\widetilde X^x_t=\widetilde b(\widetilde X_t^x)dt+\sigma(\widetilde X_t^x)d\widetilde W_t,\,\, \widetilde X^x_0=x,\label{tilde_SDE}
\end{align}
 where $W_t$ and $\widetilde W_t$ are  $m$-dimensional Wiener processes.  We let $P$ and $\widetilde P$ denote the probability measures of the respective spaces where the SDEs are defined.
 
  Here we think of $b:\mathbb{R}^n\to\mathbb{R}^n$ and  $\sigma:\mathbb{R}^n\to\mathbb{R}^{n\times m}$ as the measurable drift and diffusion for the base process, and we assume the modified drift has the form $\widetilde b=b+\sigma\beta$ for some measurable $\beta:\mathbb{R}^n\to\mathbb{R}^m$.

\item $X^x_t$ and $\widetilde X^x_t$ are jointly continuous in $(t,x)$.
 \item $X^x_t$ satisfies the following flow property:\\
For any bounded, measurable $G:C([0,\infty),\mathbb{R}^n)\to\mathbb{R}$, we have
\begin{align}\label{Markov_flow_property}
E_P(G(X^x_{t+\cdot})|\mathcal{F}_t)=E_P\left[G\left(X^{(\cdot)}\right)\right]\circ X^x_t.
\end{align}
\item $X^x_t$  and $\beta$ satisfy the Novikov condition
\begin{align}
E_P\left[\exp\left(\frac{1}{2}\int_0^T \|\beta(X_s^x)\|^2ds\right)\right]<\infty
\end{align}
for all $x\in\mathbb{R}^n$, $T>0$.
\item For every $T>0$, solutions to \req{tilde_SDE} satisfy uniqueness in law, up to time $T$.
\end{enumerate}

\end{assumption}
Given this, we define $P^x=(X^x)_* P$ and $\widetilde P^x=(\widetilde X^x)_*\widetilde P$ i.e. the distributions on path space, with the Borel sigma algebra:
\begin{align}
(\Omega,\mathcal{F},\mathcal{F}_t)=(C([0,\infty),\mathbb{R}^n),\mathcal{B}(C([0,\infty),\mathbb{R}^n)),\sigma(\pi_s,s\leq t)),
\end{align}
where $\pi_t$ is evaluation at time $t$.  Finally, define $X_t\equiv \pi_t$.  One can easily show that the above properties are sufficient to guarantee that Assumption \ref{general_process_assump} holds.

\begin{remark}\label{drift_diffusion_remark}
The existence of flows of solutions $X^x_t$ and $\widetilde X^x_t$ that satisfy the above conditions is guaranteed, for example, if $b$ and $\sigma$  satisfy a linear growth bound
\begin{equation}
\|b(x)\|^2+\|\sigma(x)\|^2\leq K^2(1+\|x\|^2),
\end{equation}
and the following local Lipschitz bound:\\
For each $\ell$ there exists $K_\ell$ such that 
\begin{align}
\|b(x)-b(y)\|+\|\sigma(x)-\sigma(y)\|\leq K_\ell\|x-y\|
\end{align}
on $\|x\|,\|y\|\leq \ell$, and if $\beta:\mathbb{R}^n\to\mathbb{R}^m$ is also bounded and locally Lipschitz.
\end{remark}

Fixing $T>0$, Girsanov's theorem allows one to bound the relative entropy, $R(\widetilde P^x_T||P^x_T)$, that appears in the UQ bound \req{standard_goal_oriented_bound}. See the supplementary materials to \cite{DKPP} for more details:
\begin{lemma}\label{lemma:SDE_rel_ent}
Under Assumption \ref{assump:SDE}, and given initial distributions $\mu$ and $\widetilde{\mu}$ for the base and alternative models respectively, we have
\begin{align}
H_T(\widetilde P^{\widetilde\mu}||P^\mu)\leq& \frac{1}{T} R(\widetilde\mu||\mu)+\int\left( \frac{1}{2T}\int_0^T E_{\widetilde P}\left[  \|\beta(\widetilde X^x_s)\|^2 \right]ds\right)\widetilde\mu(dx).\notag
\end{align}
\end{lemma}

\subsection{Example: Euler-Maruyama Methods for SDEs with Additive Noise}
As the final example, we consider SDEs with additive noise, approximated by a (generalized) Euler-Maruyama (EM) method.
\begin{assumption}\label{assump:SDE2}
Let $W_t$ be an $n$-dimensional Wiener process on filtered probability spaces satisfying the usual conditions,  $b:\mathbb{R}^n\to\mathbb{R}^n$ satisfy the linear boundedness and local Lipschitz properties as described in Remark \ref{drift_diffusion_remark}, and $X^x_t$ be the strong solutions  to the SDEs
\begin{align}\label{eq:SDE_base}
&dX^x_t=b(X_t^x)dt+dW_t,\,\, X^x_0=x.
\end{align}
Recall that versions can be chosen so that $X^x_t$ is jointly continuous in $(t,x)$ and $X^x_t$ satisfies the flow property \req{Markov_flow_property}.

We fix $\Delta t>0$ and assume we are given a measurable vector field $\widetilde b_{\Delta t}:\mathbb{R}^n\to\mathbb{R}^n$ (the drift for the generalized EM method).  We define the approximating process $\widetilde X^x_0=x$,
\begin{align}
\widetilde X^x|_{(j\Delta t,(j+1)\Delta t]}(t)=\widetilde X^x_{j\Delta t}+\widetilde b_{\Delta t}(\widetilde X^x_{j\Delta t})(t-j\Delta t)+W_t-W_{j\Delta t}\text{ for $j\in\mathbb{Z}_0$.  
}
\end{align}
\end{assumption}
We emphasize that, for the purposes of employing the theory we have developed (i.e., to employ functional inequalities satisfied by the generator of \req{eq:SDE_base}), it is necessary to  extend $\widetilde X^x_t$ to all $t\geq 0$, and not just define it at the mesh points $j\Delta t$.  

Let $P$ denote the probability measure  on the space where the SDE is defined. Similarly to the previous example, we define $P^x=(X^x)_* P$ and $\widetilde P^x=(\widetilde X^x)_*P$, probability measures on
\begin{align}
(\Omega,\mathcal{F},\mathcal{F}_t)=(C([0,\infty),\mathbb{R}^n),\mathcal{B}(C([0,\infty),\mathbb{R}^n)),\sigma(\pi_s,s\leq t)).
\end{align}
  Assumption \ref{assump:SDE2} is sufficient to guarantee that Assumption \ref{general_process_assump} holds. The chain rule for relative entropy (see Theorem C.3.1 in \cite{dupuis2011weak}) can  be used to obtain
\begin{align}
R(\widetilde P^{\widetilde\mu}_T||P^\mu_T)\leq R(\widetilde\mu||\mu)+\int R(\widetilde P^x_T||P^x_T)\widetilde\mu(dx).
\end{align}

Let $T=N\Delta t$ for $N\in\mathbb{Z}^+$. For the purposes of bounding the relative entropy term
\begin{align}
R(\widetilde P^x_{T}||P^x_T)=R((\widetilde X^x|_{[0,N\Delta t]})_* P||(X^x|_{[0,N\Delta t]})_*P),
\end{align}
it will be useful to define the Polish space $\mathcal{Y}\equiv C([0,\Delta t],\mathbb{R}^n)$ and  the following one-step transition probabilities for a discrete-time Markov process on $\mathcal{Y}$:
\begin{align}
q(y,B)=P\left(X^{y(\Delta t)}|_{[0,\Delta t]}\in B\right),\,\,\widetilde q(y,B)=P\left(\widetilde X^{y(\Delta t)}|_{[0,\Delta t]}\in B\right).
\end{align}
Letting $\otimes_1^N q$ denote the composition on $\mathcal{Y}^N$, the Markov property implies
\begin{align}
\otimes_1^N q(x,\cdot)=\left( X^x|_{[0,\Delta t]},X^x|_{\Delta t+[0,\Delta t]},...,X^x|_{(N-1)\Delta t+[0,\Delta t]}\right)_*P
\end{align}
for all $x\in\mathbb{R}^n$, and similarly for $\widetilde q$, $\widetilde X^x$.

Therefore, using the  chain rule for relative entropy again, we obtain
\begin{align}
&R(\widetilde P^x_{N\Delta t}||P^x_{N\Delta t})= \sum_{j=0}^{N-1}\int R\left(\widetilde q(y,\cdot)||q(y,\cdot)\right) \widetilde q^{j}(x,dy).
\end{align}
for all $x\in\mathbb{R}^n$. Hence we arrive at:
\begin{lemma}\label{lemma:EM_rel_ent_expansion}
\begin{align}
&R(\widetilde P^x_{N\Delta t}||P^x_{N\Delta t})=\sum_{j=1}^N E_P\left[  R\left( \widetilde P^{(\cdot)}_{\Delta t}||P^{(\cdot)}_{\Delta t}\right)\circ \widetilde X^x_{(j-1)\Delta t}\right].
\end{align}
\end{lemma}

The one-step relative entropy can be bounded via Girsanov's theorem, similarly to Lemma \ref{lemma:SDE_rel_ent}; on each  time interval of length $\Delta t$, the tilde process is simply the solution to an SDE with constant drift and additive noise.

\begin{lemma}\label{lemma:EM_rel_ent_bound}
Under Assumption \ref{assump:SDE2}
\begin{align}
&H_{N\Delta t}(\widetilde P^{\widetilde{\mu}}||P^\mu)\leq \frac{1}{N\Delta t}R(\widetilde\mu||\mu)\\
&+\frac{1}{N}\sum_{j=1}^N \int\int  E_{ P}\left[\frac{1}{2\Delta t}\int_0^{\Delta t}  \|\widetilde b_{\Delta t}(y)-b(\widetilde X^x_s)\|^2ds \right] \widetilde p^{\Delta t}_{j-1}(x,dy)\widetilde\mu(dx),\notag
\end{align}
where $\widetilde p^{\Delta t}_j(x,dy)=(\widetilde X^x_{j\Delta t})_*P$. 
\end{lemma}

\subsubsection{Euler-Maruyama  Error Bounds}
We end this section by specializing the results to the Euler-Maruyama method, $\widetilde b_{\Delta t}\equiv b$.

If we assume  $b$ is $C^1$ with bounded first derivative and $Db$ is $L$-Lipschitz then Taylor expanding $b$ gives
\begin{align}
\int_0^{\Delta t} E_{ P}&\left[  \|\widetilde b_{\Delta t}(y)-b(\widetilde X^y_s)\|^2 \right]ds\leq \tr\left(Db(y)Db(y)^T\right)\frac{\Delta t^2}{2}\\
&+\|Db(y)b(y)\|^2\frac{\Delta t^3}{3}+ \frac{16\sqrt{2}\Gamma((n+3)/2)}{5\Gamma(n/2)}L\|Db\|_\infty \Delta t^{5/2}\notag\\
&+\frac{2n(n+2)L^2}{3} \Delta t^3+L\|Db\|_\infty  \|b(y)\|^3\Delta t^4+\frac{ 2L^2}{5} \|b(y)\|^4\Delta t^5,\notag
\end{align}
and therefore
\begin{align}\label{EM_H_bound2}
H_{N\Delta t}\leq& \frac{1}{N\Delta t}R(\widetilde\mu||\mu)+\frac{\Delta t}{4}\frac{1}{N}\sum_{j=1}^N \int E_P\left[   \|Db(\widetilde X^x_{(j-1)\Delta t})\|_F^2\right]\widetilde\mu(dx)\\
&+\Delta t^{3/2}\bigg(  \frac{8\sqrt{2}\Gamma((n+3)/2)}{5\Gamma(n/2)}L\|Db\|_\infty +\frac{n(n+2)L^2}{3} \Delta t^{1/2}\notag\\
&+\frac{1}{N}\sum_{j=1}^N \int E_P\bigg[\frac{\Delta t^{1/2}}{6}\|Db(\widetilde X^x_{(j-1)\Delta t})b(\widetilde X^x_{(j-1)\Delta t})\|^2\notag\\
&+\frac{L\|Db\|_\infty}{2}  \|b(\widetilde X^x_{(j-1)\Delta t})\|^3\Delta t^{3/2}+\frac{ L^2}{5} \|b(\widetilde X^x_{(j-1)\Delta t})\|^4\Delta t^{5/2} \bigg]\widetilde\mu(dx)\bigg),\notag
\end{align}
where $\|\cdot\|_F$ denotes the Frobenius matrix norm.

This is not the tightest possible bound and alternatives can be obtained by Taylor expanding further, but it gives an idea of the type of result that can be obtained under various smoothness assumptions on $b$.

If  the initial distributions have the form $d\widetilde\mu=e^{-\widetilde\phi}dx$  and $d\mu=e^{-\phi}dx$, where $\widetilde\phi$ and $\phi$ are known functions, then the relative entropy term takes the  form
\begin{align}\label{ics_rel_ent}
R(\widetilde\mu||\mu)=\int (\phi(x)-\widetilde\phi(x))   e^{-\widetilde\phi(x)}dx.
\end{align}
If one can efficiently   sample from $\widetilde\mu$ then \req{EM_H_bound2}  and \req{ics_rel_ent}  can be estimated via Monte Carlo methods, providing UQ bounds that involve a mixture of a priori and a posteriori data.

\appendix

\setcounter{theorem}{0}
    \renewcommand{\thetheorem}{\Alph{section}\arabic{theorem}}
\setcounter{proposition}{0}
    \renewcommand{\theproposition}{\Alph{section}\arabic{proposition}}

\setcounter{lemma}{0}
    \renewcommand{\thelemma}{\Alph{section}\arabic{lemma}}

\setcounter{corollary}{0}
    \renewcommand{\thecorollary}{\Alph{section}\arabic{corollary}}

\section{Proof of the Perturbation Bound}\label{app:Hilbert_space_bounds}

\begin{lemma}
Let $H$ be  a Hilbert space, $A:D(A)\subset H\to H$ be a linear operator, and $B:H\to H$ be a bounded self-adjoint operator.  Suppose there exist $D>0$ and $x_0\in H$ with $\|x_0\|=1$ such that
\begin{align}
\langle Bx_0,x_0\rangle=0 \,\,\,\text{ and }\,\,\, \Real(\langle Ax,x\rangle)\leq -D\|P^\perp x\|^2
\end{align}
for all $x\in D(A)$,  where $P^\perp$ is the orthogonal projector onto $x_0^\perp$.

Define
\begin{align}\label{B+_def}
B^+\equiv \max\left\{\sup_{\|y\|=1}\langle By,y\rangle,0\right\}.
\end{align}
Then for any  $0\leq c <D/B^+$ we have
\begin{align}\label{perturb_bound2}
\sup_{x\in D(A),\|x\|=1}\Real(\langle (A+c B)x,x\rangle)\leq\frac{c^2\|Bx_0\|^2}{D-c B^+}.
\end{align}

\end{lemma}
\begin{proof}
Let $x\in D(A)$ with $\|x\|=1$.  Define $a =\langle x_0,x\rangle$.  (Here we will use the convention of linearity in the second argument).  We have $\|P^\perp x\|^2=1-|a |^2$, and so $|a |\leq 1$ with equality if and only if $P^\perp x=0$.

We can decompose $x=a  x_0+\sqrt{1-|a |^2}v$, where either  $v=0$ and $|a |=1$ if $P^\perp x=0$ or $v=P^{\perp}x/\sqrt{1-|a |^2}$ and $\|v\|=1$ if $P^\perp x\neq 0$.  In either case, $v\perp x_0$.

With this, we have
\begin{align}
&\sup_{x\in D(A),\|x\|=1}\Real(\langle (A+c B)x,x\rangle)=\sup_{x\in D(A),\|x\|=1}\left\{ \Real(\langle Ax,x\rangle)+c \Real(\langle Bx,x\rangle)\right\}\\
\leq&  \sup_{\beta\in[0,1]}\left\{ -D(1-\beta^2)+2c \Real(\langle\sqrt{1-\beta^2}v,a  Bx_0\rangle)+c(1-\beta^2)\langle Bv,v\rangle\right\}\notag\\
\leq  & \sup_{\beta\in[0,1]}\left\{ 2c\beta\sqrt{1-\beta^2}\| Bx_0\|-\left(D-c B^+\right)(1-\beta^2)\right\},\notag
\end{align}
where $B^+$ is given by \req{B+_def}.

Restricting to $0\leq c <D/B^+$, if $\|Bx_0\|=0$ then the supremum is $0$ and we have the result.  Otherwise, the supremum is positive and  we can use $\beta\leq 1/\beta$ and then change variables in the supremum to $r=\sqrt{1-\beta^2}/\beta$, thereby obtaining
\begin{align}
&\sup_{x\in D(A),\|x\|=1}\Real(\langle (A+c B)x,x\rangle)\\
\leq  & \sup_{\beta\in(0,1]}\left\{\beta\sqrt{1-\beta^2}\left( 2c\| Bx_0\|-\left(D-c B^+\right)\sqrt{1-\beta^2}/\beta\right)\right\},\notag\\
\leq  & \sup_{r\geq0}\left\{ 2c\| Bx_0\|r-\left(D-c B^+\right)r^2\right\}=\frac{\|Bx_0\|^2c^2}{D-c B^+}.\notag
\end{align}
\qed
\end{proof}

The previous lemma is closest in spirit to the probabilistic application, as $\|Bx_0\|^2$ plays the role of the variance.  However, one can work with non-self-adjoint perturbations, if one instead uses the definition
\begin{align}
B^+\equiv\max\left\{\sup_{\|y\|=1}\Real(\langle By,y\rangle),0\right\}
\end{align} 
and makes the replacement $\|Bx_0\|\to \|(B+B^*)x_0/2\|$ in \req{perturb_bound2}. The proof is  similar.

\section{$F$-Sobolev Inequalities }\label{app:F_sobolev}
Proposition \ref{Kac_log_sobolev} can be generalized to the $F$-Sobolev case; see the proof of Theorem 2.3 in \cite{cattiaux_guillin_2008}:
\begin{proposition}\label{Kac_F_sobolev}
Let  $A$ be the generator of $\mathcal{P}_t$ and $\mu^*$ be an invariant measure.  Suppose we have a function $F:(0,\infty)\to\mathbb{R}$ satisfying the following:
\begin{enumerate}
\item $F$ is strictly increasing,
\item $F$ is concave (hence continuous),
\item $F(1)=0$,
\item $F(x)\to\infty$ as $x\to\infty$,
\item $F(xy)\leq F(x)+F(y)$ for all $x,y\geq 0$.
\end{enumerate}
(Note that this implies $F^{-1}:(F(0^+),\infty)\to (0,\infty)$ exists, is increasing, convex, and continuous.)

Assume the $F$-Sobolev inequality holds for $\mu^*$:
\begin{align}\label{F_sobolev_def}
\int g^2 F(g^2)d\mu^*\leq -\int A[g]gd\mu^*\,\,\text{ for all $g\in D(A,\mathbb{R})$ with $\|g\|_{L^2(\mu^*)}=1$.}
\end{align}

Finally, suppose that $V\in L^1(\mu^*)$ with $V>F(0^+)$ and $\int F^{-1}(V)d\mu^*<\infty$. Then  $\mathcal{P}_t^V:L^2(\mu^*)\to L^2(\mu^*)$, defined by
\begin{align}
\mathcal{P}_t^V[g](x)=E^x\left[g(X_t)\exp\left(\int_0^t V(X_s)ds\right)\right],
\end{align}
are well-defined linear operators and the operator norm satisfies the bound
\begin{align}
\|\mathcal{P}_t^V\|\leq \exp\left[ tF\left(\int F^{-1}(V)d\mu^*\right)\right].
\end{align}

\end{proposition} 
Note that if $F(0^+)=-\infty$ then certain unbounded observables are allowed, namely those that satisfy the integrability condition \req{sobolev_int_condition}.

This proposition leads to a UQ bound of the form, \req{standard_goal_oriented_bound}.  The proof is analogous to the $\log$-Sobolev case from Section \ref{sec:log_sobolev}.
\begin{theorem}
In addition to Assumption \ref{general_process_assump}, assume the $F$-Sobolev inequality, \req{F_sobolev_def}, holds for some function, $F$, having the properties listed in Proposition \ref{Kac_F_sobolev}, $f\in L^1(\mu^*,\mathbb{R})$, and there exists $c_-<0<c_+$ such that, for all $c\in(c_-,c_+)$:
\begin{align}
F(0^+)<\pm c(f-\mu^*[f]),\,\,\,\int F^{-1}\left(\pm c\left( f-\mu^*[f]\right)\right)d\mu^*<\infty.
\end{align}

Then  a UQ bound of the form \req{standard_goal_oriented_bound} holds with
\begin{align}\label{F_Sobolev_lambda}
\Lambda(c)=\left\{ \begin{array}{cl} F\left(\int F^{-1}(V_{ c})d\mu^*\right) & \textrm{ if } c\in(c_-,c_+)  \\  + \infty & \textrm{ otherwise}
\end{array} \,.
\right. 
\end{align}

In addition, if $\Var_{\mu^*}[f]>0$, $F$ and $F^{-1}$ are smooth, $F^\prime(1)>0$, $(F^{-1})^{\prime\prime}(0)>0$, and $c\to\mu^*[F^{-1}( V_{ c})]$ is smooth on a neighborhood of $0$ and  can be differentiated under the integral then  \req{asymp_bound} holds with
\begin{align}
\Lambda^{\prime\prime}(0)=F^\prime\left(1\right)(F^{-1})^{\prime\prime}(0) \Var_{\mu^*}[f],\,\,\, \eta=\frac{1}{T}R(\widetilde P^{\widetilde{\mu}}_T||P^{\mu^*}_T).
\end{align}
\end{theorem}

\section{Continuous-Time Jump Processes on General State Spaces}\label{app:jump}
As discussed in Section \ref{sec:discrete}, to apply our UQ bounds to the invariant measure of discrete-time Markov processes, $\mathcal{P}$ and $\widetilde{P}$, one needs to construct an ancillary continuous-time Markov process  with generators $\mathcal{P}-I$ and $\widetilde{\mathcal{P}}-I$, and also compute the associated relative entropy.  While the construction of continuous-time Markov processes from their generators is well known (see, for example,  Chapter 4.2 in \cite{ethier2009markov} or Chapter 3.3 in \cite{liggett2010continuous}), and the relative entropy computation is known in the countable-state-space case (see Section \ref{Example:CTMC}), we require a formula for the relative entropy in the general case of a Polish state space. To the best of our knowledge, this computation is new, though it closely mirrors the established results; hence we present only a short outline.

In order to obtain an explicit formula for the Radon-Nikodym derivative, and thereby compute the relative entropy, it is useful to utilize an explicit construction, as in the countable-state-space case (see, for example, Appendix 1 of \cite{kipnis2013scaling}), rather than invoking more general existence theorems:

Let $(\mathcal{X},\mathcal{B}_{\mathcal{X}})$ be a Polish space and $p(x,dy)$ be a probability kernel on $\mathcal{X}$.  Given $\lambda>0$  define the probability kernel, $p^J$,  on the Polish space $(\mathcal{X}\times (0,\infty),\mathcal{B}_{\mathcal{X}}\bigotimes\mathcal{B}_{(0,\infty)})$:
\begin{align}\label{pJ_def}
p^J((x,s),\cdot)=p(x,dy)\times \lambda e^{-\lambda t}dt.
\end{align}
For any probability measure $\pi$ on $(\mathcal{X},\mathcal{B}_{\mathcal{X}})$, let $P^\pi$  (for $\pi=\delta_x$ we simply write $P^x$) be the unique probability measure on $(\Omega,\mathcal{F})\equiv(\prod_{n=0}^\infty(\mathcal{X}\times (0,\infty)),\bigotimes_{n=0}^\infty(\mathcal{B}_{\mathcal{X}}\bigotimes\mathcal{B}_{(0,\infty)}))$ generated by the transition probabilities $p^J$ and initial distribution  $\pi\times (\lambda e^{-\lambda t}dt)$.  Also, define the jump process, jump intervals, and jump times:
\begin{align}
X^J_n\equiv\pi_1\circ \pi_n,\,\,\, \Delta_n\equiv\mu\circ\pi_n,\text{ for $n\in\mathbb{Z}_0$, }\,\, J_0\equiv 0,\,\, J_n\equiv\sum_{k=0}^{n-1}\Delta_k \text{ for $n\in\mathbb{Z}^+$,}
\end{align}
where $\pi_i$ denote projections onto components.  The jump rates are positive constants, so one obtains $J_n(\omega)\to\infty$ a.s. as $n\to\infty$.
 
$(X_n^J,\Delta_n)$ is a Markov process under  $P^\pi$ with transition probabilities $p^J$ and initial distribution $\pi\times (\lambda e^{\lambda t}dt)$. Use this to define the associated c{\`a}dl{\`a}g process
\begin{align}\label{Xt_def}
X_t(\omega)=X_n^J,\text{ where }t\in [J_n(\omega),J_{n+1}(\omega))
\end{align} 
and the probability kernels on $\mathcal{X}$,
\begin{align}\label{tilde_p_def}
p_t(x,A)\equiv P^x(X_t\in A),\,\,t\geq 0,\,\,x\in\mathcal{X}.
\end{align}
Finally,  let $\mathcal{F}_t$ be the natural filtration for $X_t$.

With this setup, we have the following theorem:
\begin{theorem}\label{theorem:markov_family}
 $(\Omega,\mathcal{F},\mathcal{F}_t,X_t,P^x)$, $x\in\mathcal{X}$, is a c{\`a}dl{\`a}g Markov family with transition probabilities $p_t$. More specifically:
 \begin{enumerate}
\item $(\Omega,\mathcal{F},\mathcal{F}_t)$, $t\geq 0$ is a filtered probability space and $X_t$ is an $\mathcal{X}$-valued, $\mathcal{F}_t$-adapted, c\`adl\`ag  process.
\item $p_t(x,dy)$, $t\geq 0$, are time homogeneous transition probabilities on $\mathcal{X}$.
\item $P^x$, $x\in\mathcal{X}$ are probability measures with $(X_0)_*P^x=\delta_x$ for each $x\in\mathcal{X}$.
\item For every measurable set $F$, $x\to P^x(F)$ is universally measurable.
\item For each $x\in\mathcal{X}$, $P^x(X_{t+s}\in B|\mathcal{F}_s)=p_t(X_s,B)$ $P^x$-a.s. In particular, $p_t(x,B)=P^x(X_t\in B)$.
\end{enumerate}

\end{theorem}

One also obtains realizability of the semigroup $\exp(t\lambda(\mathcal{P}-I))$ by a  probability kernel:
\begin{theorem}\label{thm:markov}
If $\mu^*$ is an invariant measure for $p$ then $\mu^*$ is invariant for ${p}_t$ for all $t\geq 0$ and  the bounded linear operators on $L^2(\mu^*)$,
\begin{align}
\mathcal{P}[f](x)\equiv \int f(y)p(x,dy),\,\,\,{\mathcal{P}}_t[f](x)\equiv\int f(y) p_t(x,dy),
\end{align}
satisfy
\begin{align}
{\mathcal{P}}_t=\exp(t\lambda(\mathcal{P}-I)),
\end{align}
for all $t\geq 0$, where the right-hand side is the operator exponential for bounded operators on $L^2(\mu^*)$.
\end{theorem}
These results are all straightforward to prove by using the same strategy as the discrete-state-space case.  

The formula for the Radon-Nikodym derivative for two measures constructed as above is also straightforward; the only complication is that here, the jump chain $(X_n^J,\Delta_n)$ is generally not recoverable from $X_t$; specifically, the $J_n$ are not $\mathcal{F}_t$-stopping times (this is because `jumps' do not necessarily change the state, unlike the construction commonly used when the state space is discrete).  Hence, we must derive a formula for the Radon-Nikodym derivative on the enlarged filtration
\begin{align}
\mathcal{G}_t\equiv\sigma(1_{J_n\leq s},X_{J_n\wedge s}:s\leq t,n\geq 0).
\end{align}
Otherwise, the computation closely mirrors the discrete case (again, see \cite{kipnis2013scaling}) and one arrives at:
\begin{theorem}\label{thm:radon_nikodym}
Suppose we have probability measures $\widetilde\mu$, $\mu$  and probability kernels $\widetilde{p}(x,dy)$, $p (x,dy)$ on $\mathcal{X}$.  Assume that $\widetilde\mu\ll\mu$ and $\widetilde{p}(x,\cdot)\ll p (x,\cdot)$ for $\widetilde\mu$ a.e. $x$.  In particular, we have $h\in L^+(\mathcal{X}\times\mathcal{X})$ such that 
\begin{align}\label{h_def}
\widetilde{p}(x,dy)=h(x,y)p (x,dy) \text{ for }\widetilde\mu\text{ a.e. }x.
\end{align} 

Given $\lambda>0$, construct the probability measures $\widetilde{P}^{\widetilde\mu}$ and $P^{\mu}$  on $\Omega$ from $\widetilde p$ and $p$ respectively, and define the process $X_t$ as in \req{Xt_def}.

Suppose $(\widetilde{\mathcal{P}}^\dagger)^n[\widetilde\mu]\ll\widetilde\mu$ for all $n$ (in particular, if $\widetilde\mu$ is invariant for $\widetilde{p}$). Then for any $t\geq 0$ we have $\widetilde P^{\widetilde\mu}|_{\mathcal{G}_t}\ll P^{\mu}|_{\mathcal{G}_t}$ and
\begin{align}\label{eq:radon_nikodym}
\frac{d\widetilde P^{\widetilde\mu}|_{\mathcal{G}_t}}{dP^{\mu}|_{\mathcal{G}_t}}=\frac{d\widetilde\mu}{d\mu}(X_0)\prod_{n\geq 1:J_n\leq t}h(X_{J_{n-1}\wedge t},X_{J_n\wedge t}).
\end{align}
\end{theorem}
By an analogous computation to the CTMC case, \req{eq:rel_ent_CTMC}, the formula for the Radon-Nikodym derivative (\ref{eq:radon_nikodym}) leads to the following formula for the relative entropy:
\begin{theorem}\label{thm:rel_ent}
Suppose we have probability measures $\widetilde\mu$, $\mu$  and probability kernels $\widetilde{p}(x,dy)$, $p (x,dy)$ on $\mathcal{X}$.  Assume that $\widetilde\mu\ll\mu$ and $\widetilde{p}(x,\cdot)\ll p (x,\cdot)$ for $\widetilde\mu$ a.e. $x$. 

Suppose  $(\widetilde{\mathcal{P}}^\dagger)^n[\widetilde\mu]\ll\widetilde\mu$ for all $n$ (in particular, if $\widetilde\mu$ is invariant for $\widetilde{p}$).  Then for any $t\geq 0$ we have
\begin{align}
R(\widetilde P^{\widetilde\mu}|_{\mathcal{G}_t}||P^{\mu}|_{\mathcal{G}_t})=R(\widetilde\mu||\mu)+ \lambda \int_0^t \widetilde{E}^{\widetilde\mu}\left[\int\log(h(X_s,z)) h(X_s,z)p (X_s,dz)\right]ds,
\end{align}
where $h$ is as defined in \req{h_def}. 
\end{theorem}
It is also useful to note that, by the data processing inequality (see Theorem 14 in \cite{1705001}), $\mathcal{F}_t\subset\mathcal{G}_t$ implies
\begin{align}
R(\widetilde{P}^{\widetilde\mu}|_{\mathcal{F}_t}||P ^{\mu}|_{\mathcal{F}_t})\leq R(\widetilde P^{\widetilde\mu}|_{\mathcal{G}_t}||P^{\mu}|_{\mathcal{G}_t}).
\end{align}
When $\widetilde{\mu}$ is an invariant measure we obtain the following simpler formula:
\begin{corollary}\label{corollary:rel_ent_bound}
Suppose we have probability measures $\widetilde\mu^*$, $\mu$  and probability kernels $\widetilde{p}(x,dy)$, $p (x,dy)$ on $\mathcal{X}$. If $\widetilde\mu^*$ is invariant for $\widetilde{p}$ then for all $t>0$
\begin{align}
R(\widetilde{P}^{\widetilde\mu^*}|_{\mathcal{F}_t}||P^{\mu}|_{\mathcal{F}_t})\leq R(\widetilde\mu^*||\mu)+\lambda t\int R(\widetilde{p}(x,\cdot)||p (x,\cdot))d\widetilde\mu^*.
\end{align}
\end{corollary}

This is the  relative entropy bound that was used in Section \ref{sec:discrete}, when applying our UQ results to invariant measures of discrete-time Markov processes.

\subsection*{Acknowledgments}
The research of L. R.-B. was partially supported by the National Science Foundation (NSF) under the grant DMS-1515712 and the 
Air Force Office of Scientific Research (AFOSR) under the grant FA-9550-18-1-0214.

\bibliographystyle{unsrt}
\bibliography{refs}

\end{document}